\documentclass{lmsMODIFIED}

\usepackage{amsmath, amscd, amsfonts, amssymb}
\usepackage{graphicx}

\newtheorem{theorem}{Theorem}[section] 
\newtheorem{lemma}[theorem]{Lemma}     
\newtheorem{corollary}[theorem]{Corollary}
\newtheorem{proposition}[theorem]{Proposition}

\newnumbered{definition}[theorem]{Definition}
\newnumbered{remark}[theorem]{Remark}
\newnumbered{example}[theorem]{Example}

\newcommand{\abs}[1]{\vert#1\vert}

\newcommand{\Real}{\mathbb{R}}
\newcommand{\Comp}{\mathbb{C}}

\newcommand{\To}{\longrightarrow}

\newcommand{\Union}[2]{\bigcup_{#1}#2_{#1}}
\newcommand{\DisjUnion}[2]{\bigsqcup_{#1}#2_{#1}}
\newcommand{\DistSet}[3]{#1+\rho^{-1}(#2)t(#3)}
\newcommand{\ROI}[1]{O_{#1}}
\newcommand{\roi}[1]{\mathcal{O}_{#1}}
\newcommand{\less}{\setminus}
\newcommand{\Res}[1]{\overline{#1}}
\newcommand{\res}[1]{\overline{#1}}
\newcommand{\Coset}[2]{#1+t(#2)\roi{F}}
\newcommand{\g}{\gamma}
\newcommand{\calL}{\mathcal{L}}
\newcommand{\calS}{\mathcal{S}}
\newcommand{\Char}[1]{\mbox{char}_{#1}}

\newcommand{\al}{\alpha}
\newcommand{\CG}{\Comp(\Gamma)}

\newcommand{\w}{\omega}
\newcommand{\mult}[1]{#1^{\times}}
\newcommand{\dmult}[1]{d\overset{\mbox{\tiny$\times$}}{#1}}
\newcommand{\nab}{\nabla}
\newcommand{\symb}[2]{\{#1,#2\}}

\newcommand{\comment}[1]{}
\newcommand{\realpart}[1]{\mbox{Re}(#1)}

\newcommand{\frakf}{\mathfrak{f}}

\newcommand{\frakt}{\mathfrak{t}}
\newcommand{\ep}{\varepsilon}

\title[Integration on Valuation Fields over Local Fields]
 {Integration on Valuation Fields over Local Fields} 

\author{Matthew T. Morrow}

\classno{11S80 (primary), 11S40, 28C10, 20G05, 81Q30 (secondary)}

\extraline{The author is supported by an EPSRC Doctoral Training Grant at the University of Nottingham.}

\begin{document}
\maketitle

\begin{abstract}
We present elements of a theory of translation-invariant integration, measure, and harmonic analysis on a valuation field with local field as residue field. This extends the work of Fesenko. Applications to zeta integrals for two-dimensional local fields are then considered.
\end{abstract}


\subsection*{Introduction}

Integration over a valuation field is an important problem in certain areas of mathematics and mathematical physics. In the case of a higher dimensional local field it was first considered by Fesenko \cite{Fesenko-analysis-on-arithmetic-schemes} \cite{Fesenko-analysis-on-loop-spaces}, and later by Kim and Lee \cite{Kim-Lee1}. A model theoretic approach to the theory of translation invariant integration follows from the works of Hrushovski and Kazhdan \cite{Hrushovski-Kazhdan1} \cite{Hrushovski-Kazhdan2} in the case of characteristic zero (and sufficiently large positive characteristic) residue field.

In the theory of two-dimensional arithmetic schemes a theory of translation invariant integration, harmonic analysis, and zeta integrals over two dimensional local fields is required in order to generalise the techniques of Tate \cite{Tate's-thesis} and Iwasawa \cite{Iwasawa} in the one-dimensional case. Such a theory will have applications to the main open problems in the arithmetic of elliptic curves over global fields.

To study the representation theory of algebraic groups over two-dimensional local fields (see \cite{Kim-Lee1} for measure theory, \cite{Kim-Lee2} for further references) and Langlands philosophy, such a theory is invaluable.

In mathematical physics, the Feynman integral is not understood rigorously (see \cite{Johnson-Lapidus} for discussion of the problems). The valuation fields $\Comp(t)$ and $\Real(t)$ may be identified with subspaces of the space of continuous paths, and so measure theory on them may provide insight into Feynman measure.

We now outline the content of the paper and important ideas.

Let $F$ be a valuation field with valuation group $\Gamma$ and integers $\roi{F}$, whose residue field $\res{F}$ is a non-discrete, locally compact field (ie. a local field: $\Real$, $\Comp$, or non-archimedean). Given a Haar integrable function $f:\res{F}\to\Comp$, we consider the lift, denoted $f^{0,0}$, of $f$ to $\roi{F}$ by the residue map, as well as the functions of $F$ obtained by translating and scaling \[x\mapsto f^{0,0}(x+a),\qquad x\mapsto f^{0,0}(\al x)\] for $a\in F$, $\al\in\mult{F}$. We work with the space spanned by these function as $f$ varies. A simple linear independence result (proposition \ref{prop_linear_independence_of_lifted_functions}) is key to proving that an integral taking values in $\Comp\Gamma$ (the complex group algebra of $\Gamma$), under which $f^{0,0}$ has value $\int_{\res{F}}f(u)\,du$, is well defined.

The integration yields a translation invariant measure. For example, in the case of $\Comp(t)$, the set $St^n+t^{n+1}\Comp[t]$ is given measure $\mu(S)X^n$ in $\Real(X)$, where $S$ is a Lebesgue measurable subset of $\Comp$ of positive finite measure $\mu(S)$.

The first elements of a theory of harmonic analysis are then presented for fields which are self-dual in a certain sense. Not all functions occuring belong to the space of integrable functions already considered, and so the integral is extended. A Fourier transform is defined and a double transform formula proved.

There then appears a short section on integrating over the multiplicative group of $F$. Here we generalise the relationship $\dmult{x}=\abs{x}^{-1}d^{+}x$ between the multiplicative and additive Haar measures of a local field.

If $F$ is a higher dimensional local field then the main results of the aforementioned sections reduce to results of Fesenko in \cite{Fesenko-analysis-on-arithmetic-schemes} (henceforth referred to as [AoAS]) and  \cite{Fesenko-analysis-on-loop-spaces}. However, the results here are much more general (and abstract); in particular, if $\res{F}$ is archimedean then we provide proofs of claims in \cite{Fesenko-analysis-on-arithmetic-schemes} regarding higher dimensional archimedean local fields, and whereas those papers work with complete fields, we require no topological conditions. The more abstract approach to the integral developed in this paper appears to be powerful; the author has also used it to prove Fubini's theorem for certain repeated integrals over $F\times F$ and deduce the existence of a translation invariant integral on $\mbox{GL}_n(F)$.

In the final sections of the paper, we consider various zeta integrals. Firstly, parts of the theory of local zeta integrals over $\res{F}$ are lifted to $F$. In doing so we are lead to consider certain divergent integrals related to quantum field theory and we suggest a method of obtaining epsilon constants from such integrals.

We then consider zeta integrals over the local field $\res{F}$; a modified Fourier transform $f\mapsto f^*$ is defined (following Weil \cite{Weil-fonction-zeta} and [AoAS] in the non-archimedean case) and we prove, following the approaches of Tate and Weil, that it leads to a local functional equation, with appropriate epsilon factor, with respect to $s$ goes to $2-s$: \[Z(g^*,\w^{-1},2-s)=\ep_*(\w,s)Z(g,\w,s).\] After explicitly calculating some $^*$-transforms we use this functional equation to calculate the $^*$-epsilon factors for all quasi-characters $\w$. These results on zeta integrals and epsilon factors are then used to prove that $^*$ is an endomorphism of $\calS(K)$, which, though important, appears not to have been considered before. When $\res{F}$ is archimedean we define a new $^*$-transform and consider some examples.

Zeta integrals over the two-dimensional local field are then considered following [AoAS]. Lacking a measure theory on the topological $K$-group $K_2^{\mbox{\tiny top}}(F)$ (the appropriate object for class field theory of $F$; see \cite{Fesenko-multidimensional-local-class-field-theory}), a zeta integral over (a subgroup of) $\mult{F}\times\mult{F}$ is considered: \[\zeta(f,\chi,s)=\int^{\mult{F}\times\mult{F}}f(x,y)\,\chi\circ\mathfrak{t}(x,y)\abs{\mathfrak{t}(x,y)}^s\,\Char{T}(x,y)\,\dmult{x}\dmult{y}.\] Analytic continuation and functional equation are established for certain quasi-characters; indeed in these cases the functional equation, and explicit L-functions and epsilon factors, follow from properties of the $^*$-transform on $\res{F}$. Our results are compared with [AoAS].

The advantages of our new approach to the integration theory are apparent in these chapters on local zeta integrals. Our approach is to lift known results up from the local field $\res{F}$, rather than try to generalise the proof for a local field to the two-dimensional field. For example, we therefore immediately know that many of our local zeta functions have analytic continuation. Apparently complicated integrals on $F$ reduce to familiar integrals over $\res{F}$ where manipulations are easier; for example, we may work at the level of $\res{F}$ even though we are calculating epsilon factors for \emph{two-dimensional} zeta integrals.

The appendices are used to discuss theory which would otherwise interrupt the paper. Firstly, the set manipulations in [AoAS] (use to prove that the measure is well-defined) are reproved here more abstractly. Secondly we discuss what we mean by a holomorphic function taking values in a complex vector space; this allow us to discuss analytic continuation of our zeta functions. Finally the extension of the integration theory to $F\times F$ is considered; no proofs are given and similar results may be found in \cite{Morrow_2}.

\subsection*{Notation}
Let $\Gamma$ be a totally ordered group and $F$ a field with a valuation $\nu:F^{\times}\to\Gamma$ with residue field $\Res{F}$, ring of integers $\roi{F}$ and residue map $\rho:\roi{F}\to\Res{F}$ (also denoted by an overline). Suppose further that the valuation is split; that is, there exists a homomorphism $t:\Gamma\to F^{\times}$ such that $\nu\circ t=\mbox{id}_\Gamma$. The splitting of the valuation induces a homomorphism $\eta:F^{\times}\to\Res{F}^{\times}$ by $x\mapsto \Res{xt(-\nu(x))}$. Assume also that $\Gamma$ contains a minimal positive element, denoted $1$.

Sets of the form $\Coset a{\g}$ are called \emph{translated fractional ideals}; $\g$ is referred to as the \emph{height} of the set.

$\CG$ denotes the field of fractions of the complex group algebra $\Comp\Gamma$ of $\Gamma$; the basis element of the group algebra corresponding to $\g\in\Gamma$ shall be written as $X^{\g}$ rather than as $\g$. With this notation, $X^{\g}X^{\delta}=X^{\g+\delta}$. Note that if $\Gamma$ is a free abelian group of finite rank $n$, then $\CG$ is isomorphic to the rational function field $\Comp(X_1,\dots,X_n)$.

We fix a choice of Haar measure on $\res{F}$; occasionally, for convenience, we shall assume that $\roi{\res{F}}$ has measure one. The measure on $\mult{\res{F}}$ is chosen to satisfy $\dmult{x}=\abs{x}^{-1}d^+x$.

\begin{remark*}
The assumptions above hold for a higher dimensional local field. For basic definitions and properties of such fields, see \cite{IHLF}.

Indeed, suppose that $F=F_n$ is a higher dimensional local field of dimension $n\ge2$: we allow the case in which $F_1$ is an archimedean local field. If $F_1$ is non-archimedean, instead of the usual rank $n$ valuation $\mathbf{v}:F^{\times}\to\mathbb{Z}^n$, let $\nu$ be the $n-1$ components of $\mathbf{v}$ corresponding to the fields $F_n,\dots,F_2$; note that $\mathbf{v}=(\nu_{\Res{F}}\circ\eta,\nu)$. If $F_1$ is archimedean, then $F$ may be similarly viewed as an valuation field with value group $\mathbb{Z}^{n-1}$ and residue field $F_1$.

The residue field of $F$ with respect to $\nu$ is the local field $\Res{F}=F_1$. If $F$ is non-archimedean, then the ring of integers $\ROI{F}$ of $F$ with respect to the rank $n$ valuation is equal to $\rho^{-1}(\roi{\Res{F}})$, while the groups of units $\mult{\ROI{F}}$ with respect to the rank $n$ valuation is equal to $\rho^{-1}(\roi{\Res{F}}^{\times})$.
\end{remark*}

The primary reference for this work is \cite{Fesenko-analysis-on-arithmetic-schemes}, to which we will refer as [AoAS].

\subsection*{Acknowledgements}
I am grateful for useful discussions with I. Fesenko.


\section{Integration on $F$}
In this section we explain a basic theory of integration on $F$. The following definition is fundamental:

\begin{definition}
Let $f$ be a function on $\Res{F}$ taking values in an abelian group $A$; let $a\in F$, $\g\in\Gamma$. The $\emph{lift}$ of $f$ at $a,\g$ is the $A$-valued function on $F$ defined by
\[f^{a,\g}(x)=
    \begin{cases}
    f(\Res{(x-a)t(-\g)}) & \quad x\in a+t(\g)\roi{F} \\
    0 & \quad\mbox{otherwise}
    \end{cases}
\]
\end{definition}

In other words,
\[f^{0,0}(x)=
    \begin{cases}
    f(\Res{x}) & \quad x\in\roi{F} \\
    0 & \quad\mbox{otherwise}
\end{cases}\]
and $f^{a,\g}(a+t(\g)x)=f^{0,0}(x)$ for all $x$.

It is useful to determine how lifted functions behave on translated fractional ideals:

\begin{lemma}\label{lemma_lifted_functions_1}
Let $f^{a,\g}$ be a lifted function as in the definition; let $b\in F$, $\delta\in\Gamma$. Then for all $x$ in $\roi{F}$,
\begin{description}
\item[case $\delta>\g$]
\[f^{a,\g}(b+t(\delta)x)=
    \begin{cases}
    f(\Res{(b-a)t(-\g)})&\quad b\in\Coset a{\g}\\
    0&\quad\mbox{otherwise}
    \end{cases}
\]
\item[case $\delta=\g$]
\[f^{a,\g}(b+t(\delta)x)=
    \begin{cases}
    f(\Res{(b-a)t(-\g)}+\Res{x})&\quad b\in\Coset a{\g}\\
    0&\quad\mbox{otherwise}
    \end{cases}
\]
\item[case $\delta<\g$]
\[f^{a,\g}(b+t(\delta)x)=
    \begin{cases}
    f(\Res{(b+t(\delta)x-a)t(-\g)})
        &\quad x\in(a-b)t(\delta)^{-1}+t(\g-\delta)\roi{F}\\
    0&\quad\mbox{otherwise}
    \end{cases}
\]
In particular in this final case, if $x,y\in\roi{F}$ are such that
$f^{a,\g}(b+t(\delta)x)$ and $f^{a,\g}(b+t(\delta)y)$ are
non-zero, then $\Res{x}=\Res{y}$.
\end{description}
\end{lemma}
\begin{proof}
This follows from the definition of a lifted function by direct
verification.
\end{proof}

Let $\calL$ be the space of complex-valued Haar integrable functions on $\res{F}$.

\begin{remark}\label{remarks_lifted_functions}
\mbox{}
\begin{enumerate}
\item For $a\in F,\g\in\Gamma$, let $\calL^{a,\g}$ denote the space of complex valued functions on $F$ of the form $f^{a,\g}$, for $f\in\calL$. Suppose $\Coset{a_1}{\g_1}=\Coset{a_2}{\g_2}$. Then $\g_1=\g_2$ and \[f^{a_1,\g_1}(x)=f^{a_2,\g_2}(x+a_2-a_1)=g^{a_2,\g_2}(x)\] where $g\in\calL$ is the function $g(y)=f(y+\Res{(a_2-a_1)t(-\g_2)})$. Hence $\calL^{a_1,\g_1}=\calL^{a_2,\g_2}$.

\item Given a lifted function $f^{a,\g}$ and $\tau\in F$, then the translated function $x\mapsto f^{a,\g}(x+\tau)$ is the lift of $f$ at $a-\tau,\g$
\end{enumerate}
\end{remark}

\begin{definition}
For $J=\Coset a{\g}$ a translated fractional ideal of $F$, define $\calL(J)$ to be the space of complex-valued functions of $F$ of the form $f^{a,\g}$, for $f\in\calL$.

Introduce an integral on $\calL(J)$ by
\begin{align*}
	\int^J:\calL&\to\Comp\\
	f^{a,\g}&\mapsto \int_{\res{F}}f(u)\,du.
	\end{align*}
\end{definition}

By remarks \ref{remarks_lifted_functions} and translation invariance of the Haar integral on $\calL$, the integral is well-defined (ie. independent of $a,\gamma$).

\begin{proposition} \label{prop_linear_independence_of_lifted_functions}
The spaces $\calL(J)$, as $J$ varies over all translated fractional ideals, are linearly independent.
\end{proposition}
\begin{proof}
Let $J_i$, for $i=1\dots,n$, be distinct translated fractional ideals, of height $\g_i$ say. Suppose $f_i\in\calL^{J_i}$ for each $i$, with $\sum_if_i=0$.
We prove that $f_i=0$ for all $i$ by induction on $n$. We may suppose that $\g_1\le\g_2\le\cdots\le\g_n$; write $J_i=\Coset{a_i}{\g_i}$.

If $\g_1<\g_2$  then \[f_1(a_1+t(\g_1)x)=-\sum_{i>1}f_i(a_1+t(\g_1)x)\] for all $x\in\roi{F}$. The first case of lemma \ref{lemma_lifted_functions_1} implies that $f_1$ is therefore the lift of a constant function, and so $f_1=0$ (for $\calL$ contains no other constant function). The assertion now follows by the inductive hypothesis.

If $\g_1=\g_2$ then the linear dependence relation breaks into two separate relations:
\[\sum_{i\; :\; J_i\subseteq J_1}f_i=0,\qquad\sum_{i\; :\;J_i\subseteq J_2}f_i=0.\]
The inductive hypothesis implies that $f_i=0$ for all $i$.
\end{proof}

This linear independence result clearly allows us to extend the $\int^J$, as $J$ varies over all translated fractional ideals, to a single functional:

\begin{definition}
Let $\calL(F)_{\Comp}$ be the space of complex-valued functions spanned by $\calL(J)$ for all translated fractional ideals $J$. Let $\int^F:\calL(F)_{\Comp}\to\CG$ denote the unique linear map such that if $f\in\calL(J)$ for some $J$ of height $\gamma$, then $\int^F(f)=\int^J(f)\;X^{\gamma}$.

$\calL(F)_{\Comp}$ will be referred to as the space of complex-valued integrable functions on $F$.
\end{definition}

Remarks \ref{remarks_lifted_functions} imply that $\calL(F)_{\Comp}$ is closed under translation from $F$ and that $\int^F$ is translation invariant. We will of course usually write $\int^Ff(x)\,dx$ in place of $\int^F(f)$

\begin{remark}
If $A$ were an arbitrary $\Comp$-algebra and elements $c_{\g}\in A$ were given for each $\g\in\Gamma$, we could define an $A$-valued linear translation invariant integral on $\calL(F)$ by replacing $X^{\g}$ by $c_{\g}$ in the previous definition. However, using $X^{\g}$ ensures compatibility of the integral with the multiplicative group $\mult{F}$, in that it implies the existence of an absolute value with expected properties; see lemma \ref{lemma_abs_value}.

This phenomenon also appears when extending the integration theory to $F^n$, $\mbox{M}_n(F)$, and $\mbox{GL}_n(F)$. The best results are obtained not simply by trying integrate over a single algebraic group as a closed problem, but by taking into account the action of other groups on it.
\end{remark}

\begin{remark}
We consider how $\calL(F)_{\Comp}$ and $\int^F$ depend on $t$.

Let $t'$ be another splitting of the valuation: that is, $t'$ is a homomorphism from $\Gamma$ to $\mult{F}$ with $\nu\circ t'=\mbox{id}_{\Gamma}$. Then there is a homomorphism $u:\Gamma\to\mult{\roi{F}}$ which satisfies $t(\g)=u(\g)t'(\g)$ for $\g\in\Gamma$. Let $g\in\calL$, $a\in F$, and $\g\in\Gamma$; let $f$ be the lift of $g$ at $a,\g$ with respect to $t$, and $f'$ the lift of $g$ at $a,\g$ with respect to $t'$. Thus, by definition, $f$ and $f'$ both vanish off $J=a+t(\g)\roi{F}=a+t'(\g)\roi{F}$, and for $x\in\roi{F}$, \[f(a+t(\g)x)=g(\res{x}),\qquad f'(a+t'(\g)x)=g(\res{x}).\] Therefore $f'(a+t(\g)x)=g(\res{u(\g)}^{-1}\res{x})$ and so $\int^J(f')=|\res{u(\g)}|\int g(y)\,dy=|\res{u(\g)}|\int^J(f)$.

Let $\int^{J,t'}$ (resp. $\int^{F,t'}$) denote the integral over $J$ (resp. $F$) with respect to $t'$; the previous paragraph proves that $\int^J=\abs{\res{u(\g)}}\int^{J,t'}$. Let $\sigma:\CG\to\CG$ be the $\Comp$-linear field automorphism of $\CG$ given by $\sigma(X^{\g})=|\res{u(\g)}|X^{\g}$, for $\g\in\Gamma$. Then for all $f\in\calL(F)$, the identity \[\int^Ff(x)\,dx=\sigma\left(\int^{F,t'}f(x)\,dx\right)\] follows.
\end{remark}

\begin{proposition}
If $f$ belongs to $\calL(F)_{\Comp}$, then so does $x\mapsto\abs{f(x)}$.
\end{proposition}
\begin{proof}
The statement with $\calL$ in place of $\calL(F)_{\Comp}$ is true by definition of Haar integrability; hence the statement is true for $\calL(J)$, for $J$ any translated fractional ideal.

Now suppose that $f=\sum_{i=1}^n f_i$, with $f_i\in\calL^{J_i}$ say; we shall prove the result by induction, having done the case $n=1$. We may assume that $\g_1\le\dots\le\g_n$, where $\g_i$ is the height of $J_i$.

If $J_n=J_r$ for any $r\neq n$, then $f_n+f_r\in\calL^{J_n}$, from which the result follows by the inductive hypothesis.

Now suppose $J_n\neq J_r$ for any $r\neq n$; let $a\in J$. We claim that the identity $f_r(x)=f_r(a)$ holds for all $r\neq n$ and $x\in J_n$; indeed if $\g_r=\g_n$, then both sides vanish, and if $\g_r<\g_n$ then it follows from the first case of lemma \ref{lemma_lifted_functions_1}. With the claim established, the identity \[\abs{f}=|{\sum_{i=1}^{n-1} f_i}|+ |{f_n+\sum_{i=1}^{n-1} f_i(a)}|-|{\sum_{i=1}^{n-1}f_i(a)}|\] is easily verified, and so the proof will be complete if we can show that \[|{f_n+\sum_{i=1}^{n-1} f_i(a)}|-|{\sum_{i=1}^{n-1}f_i(a)}|\tag{$\ast$}\] belongs to $\calL(F)_{\Comp}$. Write $f_n=g^{a,\g_n}$ for some $g\in\calL$; then the function ($\ast$) is the lift at $a,\g_2$ of the Haar integrable function $|{g+\sum_{i=1}^{n-1} f_i(a)}|-|{\sum_{i=1}^{n-1}f_i(a)}|$.
\end{proof}

Although $\calL(F)_{\Comp}$ is closed under taking absolute values, the following examples show that it that there is some unusual associated behaviour.
 
\begin{example}\label{eg_null_functions}
Introduce $f_1=\Char{\{0\}}^{0,0}$, the characteristic function of $t(1)\roi{F}$, and $f_2=-2\Char{S}^{0,\g}$ where $S$ is a Haar measurable subset of $\res{F}$ with measure $1$ and $\g$ is a positive element of $\Gamma$. Let $f=f_1+f_2$.
\begin{enumerate}
\item Firstly we claim that the following hold: \[\int^F \abs{f(x)} \,dx=0,\qquad\int^F f(x)\,dx=-2X^{\g}.\]

Indeed, the second identity is immediate from the definition of the integral. For the first identity, note that as in the proof of the previous proposition (with $n=2$), \[\abs{f}=|f_1|+\abs{f_2+f_1(0)}-\abs{f_1(0)}.\] Further, $f_1(0)=1$ and the function $|f_2+1|$ is identically $1$. So $|f|=\Char{\{0\}}^{0,0}$, from which the first identity follows.

\item Secondly, the considerations above imply \[\int^F \abs{f(x)}\,dx=\int^F\abs{f_1(x)}\,dx=0,\qquad\int^F \abs{f(x)-f_1(x)}\,dx=-2X^{\g}\]

\item Finally, consider the translated function $f'(x)=f(x-a)$, where $a$ is any element of $F$ not in $\roi{F}$. Then $f'$ and $f$ have disjoint support and so
\begin{align*}
\int^F\abs{f(x)-f'(x)}\,dx
    &=\int^F\abs{f(x)}+\abs{f'(x)}\,dx\\
    &=\int^F\abs{f(x)}\,dx+\int^F\abs{f'(x)}\,dx=0
\end{align*}
by translation invariance of the integral. Also, $\int^F f(x)-f'(x)\,dx=0$. Thus $g=f-f'$ provides an example of a complex-valued integrable function on $F$ such that $\int^F \abs{g(x)}\,dx=\int^F g(x)\,dx=0$, but where the components of $g$ in each $\calL(J)$ are lifts of non-null functions.
\end{enumerate}
\end{example}

As will become apparent, it is more natural to integrate $\CG$-valued functions on $F$ than complex-valued ones, so we define our main class of functions as follows:

\begin{definition}
A $\CG$-valued function on $F$ will be said to be \emph{integrable} if and only if it has the form $x\mapsto\sum_i f_i(x)\,p_i$ for finitely many $f_i\in\calL(F)_{\Comp}$ and $p_i\in\CG$. The integral of such a function is defined to be \[\int^F f(x)\,dx= \sum_i\int^F f_i(x)\,dx\,p_i.\] This is well defined. The $\CG$ space of all such functions will be denoted $\calL(F)$; the integral is a $\CG$-linear functional on this space.
\end{definition}

In other words, $\calL(F)=\calL_{\Comp}(F)\otimes_{\Comp}\CG$ and the integral is extended in the natural way. 

The integrable functions which are complex-valued are precisely $\calL(F)_{\Comp}\subset\calL(F)$, so there is no ambiguity in the phrase 'complex-valued integrable function'.

For the sake of completeness, we summarise this section as follows:

\begin{proposition}
$\calL(F)$ is the smallest $\CG$ space of $\CG$-valued functions on $F$ which contain $g^{a,\g}$ for all $g\in\calL$, $a\in F$, $\g\in\Gamma$. There is a (necessarily unique) $\CG$-linear functional $\int^F$ on $\calL(F)$ which satisfies \[\int^F g^{a,\g}(x)\,dx=\int_{\res{F}} g(u)\,du\,X^{\g}.\]

$\calL(F)$ is closed under translation and $\int^F$ is translation invariant.
\end{proposition}

\subsection{Generalisations}
Let us discuss several extensions of the theory.

\subsubsection*{Abstraction}
Examination of the proofs in this section leads to the the following abstraction of the theory:

Let $F',\nu',t',\Gamma'$ satisfy the same conditions as $F,\nu,t,\Gamma$, except that we do not suppose $\res{F}'$ is a local field. Let $L$ be an arbitrary field, and $\calL'$ an $L$ space of $L$-valued functions on $\res{F}'$, equipped with an $L$-linear functional $I$, with the following properties:
\begin{enumerate}
\item $\calL'$ is closed under translation from $\res{F}'$ and $I$ is translation invariant (ie. $f\in\calL'$ and $a\in\res{F}'$ implies $y\mapsto f(y+a)$ is in $\calL'$ with image under $I$ equal to $I(f)$).
\item $\calL'$ contains no non-zero constant functions.
\end{enumerate}

Let $\calL'(F')$ be the smallest $L(\Gamma')$ space of $L(\Gamma')$-valued functions on $F$ which contains $f^{a,\g}$ for $f\in\calL'$, $a\in F'$, $\g\in\Gamma'$. Then there is a (necessarily) unique $L(\Gamma')$-linear functional $I^{F'}$ on $\calL'(F')$ which satisfies $I^{F'}(f^{a,\g})=I(f)\,X^{\g}$. Further, the pair $\calL'(F'), I^{F'}$ satisfy (i) and (ii) with the field $L(\Gamma')$ in place of $L$.

\subsubsection*{Integration on $F\times F$}
See appendix \ref{appendix_product_integration}.
%
%
%

\section{Measure theory} \label{measure-theory}

In this section we recover a basic measure theory from the integration theory; results of [AoAS] are reproduced and extended.

\begin{definition}
A \emph{distinguished} subset of $F$ is a set of the form $\DistSet aS{\g}$, where $a\in F$, $\g\in\Gamma$, and $S$ is a subset of $\Res{F}$ of finite Haar measure. $\g$ is said to be the $\emph{level}$ of the set.

Let $D$ denote the set of all distinguished subsets of $F$; let $R$ denote the ring of sets generated by $D$ (see the appendix for the definition of 'ring').
\end{definition}

\begin{remark}\label{rem_char_fns1}
Note that the characteristic function of a distinguished set $\DistSet aS{\g}$ is precisely the lift of the characteristic function of $S$ at $a,\gamma$. Proposition \ref{prop_linear_independence_of_lifted_functions} proves that if $\DistSet{a_1}{S_1}{\g_1}=\DistSet{a_2}{S_2}{\g_2}$, then $\g_1=\g_2$ and $S_1$ is a translate of $S_2$. In particular, the level is well defined.
\end{remark}

\begin{lemma}
Let $A_i=\DistSet {a_i}{S_i}{\g_i}$, $i=1,2$, be distinguished sets with non-empty intersection.
\begin{itemize}
\item[-] If $\g_1=\g_2$, then $A_1\cap A_2$ and $A_1\cup A_2$ are distinguished sets of level $\g_1$.
\item[-] If $\g_1\neq\g_2$, then $A_1\subseteq A_2$ if $\g_1>\g_2$, and $A_2\subseteq A_1$ if $\g_2>\g_1$.
\end{itemize}
\end{lemma}
\begin{proof}
This is immediate from the definition of a distinguished set.
\end{proof}

Referring to the appendix, it has just been shown that $D$ is a d-class of sets. By proposition \ref{prop_ddd}, the characteristic function of any set in $R$ may be written as the difference of two sums, each of characteristic functions of sets in $D$; therefore the characteristic function of any set in $R$ belongs to $\calL(F)_{\Comp}$.

\begin{definition}
Define the measure $\mu^F(W)$ of a set $W$ in $R$ by \[\mu^F(W)=\int^F \Char{W}(x)\,dx.\]
\end{definition}

By the properties of the integral, $\mu^F$ is a translation-invariant finitely additive set function $R\to\Real\Gamma$ (the real group algebra of $\Gamma$). For a distinguished set $A=\DistSet aS{\g}$, remark \ref{rem_char_fns1} implies \[\mu^F(A)=\int^F(\Char{A})=\int^F(\Char{S}^{a,\g})=\mu(S)X^{\g},\] where $\mu$ denotes our choice of Haar measure on $\res{F}$. The following examples demonstrate the unusual behaviour of this measure:

\begin{example}
\mbox{}
\begin{enumerate}
\item For $\g\in\Gamma$, the set $t(\g)\roi{F}=t(\g-1)\rho^{-1}(\{0\})$ is distinguished, with measure zero.
\item Let $S$ be a subset of $\res{F}$ of finite measure. The set $\rho^{-1}(\res{F}\setminus S)=\roi{F}\less\rho^{-1}(S)$ belongs to $R$ and has measure $-\mu(S)$. Compare this with example \ref{eg_null_functions}
\item $\mu^F$ is not countably additive. Indeed, write $\res{F}$ as a countable disjoint union of sets of finite measure; $F=\bigsqcup_i S_i$ say. Then $\roi{F}=\bigsqcup_i\rho^{-1}(S_i)$ has measure zero, while $\sum_i\mu^F(\rho^{-1}(S_i))=\infty$.
\item Suppose that $\res{F}=\mathbb{R}$. Set $A_{2n-1}=nt(-1)+\rho^{-1}([0,1/n])$ and $A_{2n}=nt(-1)+\rho^{-1}(\mathbb{R}\less[0,1/n])$ for all natural $n$. Then $\mu^F(A_{2n-1})=1/n$, $\mu^F(A_{2n})=-1/n$, and $\bigsqcup_i A_i=\bigsqcup_n nt(-1)+\roi{F}=t(-1)\rho^{-1}(\mathbb{N})$ has measure $0$.

The series $\sum_i\mu^F(A_i)$ is conditionally convergent in $\mathbb{R}$ (ie. convergent, but not absolutely convergent). By a theorem of Riemann (see eg. \cite[chapter 8.18]{Apostol}), there exists, for any real $q$, a permutation $\sigma$ of $\mathbb{N}$ such that $\sum_i \mu^F(A_{\sigma(i)})$ converges to $q$. But regardless of the permutation, $\mu^F(\bigsqcup_i A_{\sigma(i)})=0$.
\end{enumerate}
\end{example}

Let us consider a couple of examples in greater detail and give a more explicit description of the meausure:

\begin{example}\mbox{}
\begin{enumerate}
\item Suppose that $F$ is an $n$-dimensional non-archimedean local field, with local parameters $t_1,\dots,t_n$. Recall that we view $F$ as a valued field over the local field $\res{F}=F_1$, rather than over the finite field $F_0$. The results of this section prove the existence of a finitely additive set function $\mu^F$ on the appropriate ring of sets, taking values in $\mathbb{R}[X_2,\dots,X_n]$, which satisfies \[\mu^F(a+t_1^{r_1}\dots t_n^{r_n}\ROI{F})=q^{-r_1}X_2^{r_2}\dots X_n^{r_n}\] for $a\in F$ and integers $r_i$. Here $\ROI{F}$ denotes the ring of integers of $F$ with respect to the rank $n$ valuation.

However, we have not made use of any topological property of $F$; in particular, this result holds for an arbitrary field with value group $\mathbb{Z}^{n-1}$ and a non-archimedean local field as residue field. This measure theory therefore extends that developed in [AoAS], while also providing proofs of statements in [AoAS] for the case in which the local field is archimedean.

\item Suppose that $F=\res{F}((t))$, the field of formal Laurent series over $\res{F}$, or $F=\res{F}(t)$, the rational function field (here we write $t=t(1)$). Then a typical distinguished set has the form
\[\begin{array}{cl}
    a(t)+St^n+t^{n+1}\res{F}[[t]] & \mbox{(Laurent series case)}\\
    a(t)+St^n+t^{n+1}\res{F}[t] & \mbox{(rational functions case)}
\end{array}\]
for $a(t)\in F$, and $S\subset\res{F}$ of finite Haar measure. Such a set has measure $\mu(S)X^n$, where $\mu$ denotes our choice of Haar measure on $\res{F}$.
\end{enumerate}
\end{example}
%
%
%
%
%

\section{Harmonic analysis on $F$} \label{section_Harmonic_analysis}
In this section, we develop elements of a theory of harmonic analysis on $F$.
 
\begin{definition}
Suppose that $\psi:F\to S^1$ is a homomorphism of the additive group of $F$ into the group of complex numbers of unit modulus. Then $\psi$ is said to be a \emph{good character} of $F$ if it is trivial, or if it satisfies the following two conditions:
\begin{enumerate}
\item There exists $\frakf\in\Gamma$ such that $\psi$ is trivial on $t(\frakf)\roi{F}$, but non-trivial on $t(\frakf-1)\roi{F}$; $\frakf$ is said to be the \emph{conductor} of $\psi$.
\item The conductor $\res{\psi}$ of the additive group of $\Res{F}$ defined by $\res{\psi}(\res{x})=\psi(t(\frakf-1)x)$, for $x\in\roi{F}$, is continuous.
\end{enumerate}
The conductor of the trivial character may be said to be $-\infty$. The induced character on $\res{F}$ as in (ii) will be always be denoted $\Res{\psi}$.
\end{definition}

The definition of a good character is designed to replace the continuity assumption which would be imposed if $F$ had a topology.

\begin{example}\label{example_character_on_laurent_series}
Suppose that $F=\res{F}((t))$, the field of formal Laurent series over $\res{F}$ (here $t(1)=t$). Let $\psi_{\res{F}}$ be a continuous character of $\res{F}$. Then $\sum_i a_it^i\mapsto\psi_{\res{F}}(a_n)$ is a good character of $F$ of conductor $n+1$ and induced character $\psi_{\res{F}}$.
\end{example}

\begin{lemma}\label{lemma_translation_of_good_characters}
Suppose that $\psi$ is a good character of $F$ of conductor $\frakf$; let $a\in F$. Then $x\mapsto\psi(ax)$ is a good character of $F$, with conductor $\frakf-\nu(a)$; the character induced on $\Res{F}$ by $x\mapsto\psi(ax)$ is $y\mapsto\Res{\psi}(\eta(a)y)$.
\end{lemma}
\begin{proof}
This is easily checked.
\end{proof}

Given $\psi,a$ as in the previous lemma we will write $\psi_a$ for the translated character $x\mapsto\psi(ax)$ (and we employ similar notation for characters of $\res{F}$).

Before proceeding, we must make a simple assumption:
\begin{center}
\emph{We assume that a non-trivial good character $\psi$ exists on $F$.}
\end{center}
By the previous lemma we assume further that $\psi$ has conductor $1$, and we fix such a character for this section. With this choice of conductor, $x\in\roi{F}$ implies $\Res{\psi}(\Res{x})=\psi(x)$. We will take Fourier transforms of integrable functions $g$ on $\Res{F}$ with respect to the character $\Res{\psi}$; that is, $\hat{g}(x)=\int g(y)\Res{\psi}(xy)dy$.

Let $\calL(F,\psi)$ denote the $\CG$ space of $\CG$-valued functions on $F$ spanned by $f\psi_a$, for $f\in\calL(F)$, $a\in F$. The aim of this section is proposition \ref{prop_extn_of_integral_to_characters}, which states that the integral has a unique translation invariant extension to this space of functions.
 
\begin{remark} \label{remark_characters_of_higher_local_field}
Such a character certainly exists on a higher local field. Indeed, such a field is self-dual: If $\psi,\psi_1$ are good characters with $\psi$ non-trivial then there is $\al\in\mult{F}$ such that $\psi(x)=\psi_1(\al x)$ for all $x\in F$. For more details, see [AoAS] section 3.
\end{remark}

It is convenient for the following results to write $\calL^{\g}$ (where $\g\in\Gamma$) for the sum of the spaces $\calL(J)$ over all translated fractional ideals of height $\g$; this sum is direct by proposition \ref{prop_linear_independence_of_lifted_functions}. Note that if $f\in\calL^{\g}$ and $a\in F$ with $\nu(a)>\g$ then $f(x+a)=f(x)$ for all $x\in F$.

Certain products of an integrable function with a good character are still integrable:
\begin{lemma}\label{lemma_int_of_characters_1}
Let $J=\Coset{a}{\g}$ be a translated fractional ideal and $\al\in F$. If $\g=-\nu(\al)$, then $\psi_{\al}\Char{J}$ is the lift
of $\psi(\al a)\Res{\psi}_{\eta(\al)}$ at $a,\g$; if $\g>-\nu(\al)$, then $\psi_{\al}$ is constantly $\psi(\al a)$ on $J$.

Therefore, if $\g\ge-\nu(\al)$ and $f$ is in $\calL^{\g}$ then $f\psi_{\al}$ is also in $\calL^{\g}$.
\end{lemma}
\begin{proof}
The identities may be easily verified by evaluating at on $\Coset{a}{\g}$. The final statement follows by linearity.
\end{proof}

In contrast with the previous lemma we now consider the case $\g\le-\nu(\al)$:
\begin{lemma}\label{lemma_int_of_characters_2}
Let $\al_i,\g_i$ be finitely many ($1\le i\le m$, say) elements of $F,\Gamma$ respectively, and let $f_i\in\calL^{\g_i}$ for each $i$. Suppose further that $\nu(\al_i)<-\g_i$ for each $i$ and that $\sum_if_i\psi_{\al_i}$ is integrable on $F$. Then $\int^F\sum_if_i\psi_{\al_i}=0$.
\end{lemma}
\begin{proof}
The result is proved by induction on $m$. Let $y\in t(-\nu(\al_m))\roi{F}$ satisfy $\psi_{\al_m}(y)\neq 1$. The functions
\begin{align*}
x&\mapsto\sum_if_i(x+y)\psi_{\al_i}(x+y)=
    \sum_i\psi_{\al_i}(y)f_i(x+y)\psi_{\al_i}(x)\\
x&\mapsto\sum_i\psi_{\al_m}(y)f_i(x)\psi_{\al_i}(x)
\end{align*}
are integrable on $F$, the first having integral equal to that of $\sum_if_i\psi_{\al_i}$ by translation invariance of $\int^F$. Taking the difference of the two functions, noting that $f_m(x+y)=f_m(x)$, and applying the inductive hypothesis, obtains \[\int^F\sum_if_i(x)\psi_{\al_i}(x)=\psi_{a_m}(y)\int^F \sum_if_i(x)\psi_{\al_i}(x),\] competing the proof.
\end{proof}

The main result of this section may now be proved:
\begin{proposition}\label{prop_extn_of_integral_to_characters}
$\int^F$ has a unique extension to a translation-invariant $\CG$-linear functional on $\calL(F,\psi)$.
\end{proposition}
\begin{proof}
To prove uniqueness, suppose that $I$ is a translation-invariant $\CG$-linear functional on $\calL(F,\psi)$ which vanishes on $\calL(F)$. We claim that $I$ is everywhere zero; by linearity it suffices to check that $I$ vanishes on $f\psi_{\al}$ for $f\in\calL^{\g}$ (any $\g\in\Gamma$) and $\al\in F$. If $\g>-\nu(a)$, then $f\psi_{\al}$ is integrable by lemma \ref{lemma_int_of_characters_1} and so $I(f\psi_{\al})=0$. If $\g\le-\nu(\al)$, then let $y\in t(-\nu(\al))\roi{F}$ satisfy $\psi_{\al}(y)\neq 1$; as in lemma \ref{lemma_int_of_characters_2} the identity $I(f\psi_{\al})=\psi_{\al}(y)I(f\psi_{\al})$ follows from translation invariance of $I$. This completes the proof of uniqueness.

To prove existence, suppose first that $f\in\calL(F,\psi)$ is complex-valued, and write $f=\sum_i f_i\psi_{\al_i}$, for finitely many $\al_i\in F$, and $f_i\in\calL^{\g_i}$ say. Attempt to define \[I(f)=\sum_{i\,:\,\g_i\ge-\nu(\al_i)}\int^F f_i(x)\psi_{\al_i}(x)\,dx.\] We claim that this is well-defined. Indeed, if $f=0$, then function \[\sum_{i\,:\,\g<-\nu(\al_i)}f_i\psi_{\al_i}=-\sum_{i\,:\,\g\ge-\nu(\al_i)}f_i\psi_{\al_i}\] lies in $\calL^F$ by lemma \ref{lemma_int_of_characters_1}. By lemma \ref{lemma_int_of_characters_2}, the function has integral equal to zero, and so
\[0=\int^F\sum_{i\,:\,\g\ge-\nu(\al_i)}f_i(x)\psi_{\al_i}(x)\,dx=\sum_{i\,:\,\g\ge-\nu(\al_i)}\int^F f_i(x)\psi_{\al_i}(x)\,dx.\]
This proves that $I$ is well-defined.

$I$ extends to $\calL(F,\psi)$ by setting $I(\sum_j g_j\,X^{\g_j})=\sum_jI(g_j)\,X^{\g_j}$ for finitely many complex-valued $g_j$ in $\calL(F)$ and $\g_j$ in $\Gamma$. Translation invariance of $I$ follows from translation invariance of $\int^F$. 
\end{proof}

We shall denote the extension of $\int^F$ to $\calL(F,\psi)$ by $\int^F$.

\begin{remark}
The previous results maybe easily modified to prove that there is a unique extension of $\int^F$ to a translation-invariant $\CG$-linear function on the space spanned by $f\Psi$ for $f\in\calL(F)$ and $\Psi$ \emph{any} good character.
\end{remark}

\begin{example}
Compare with [AoAS] section 7, example. Suppose that $\Res{F}$ is non-archimedean, with prime $\pi$ and residue field of cardinality $q$. Let $w=(\nu_{\Res{F}}\circ\eta,\nu)$ be the valuation on $F$ with value group $\mathbb{Z}\times\Gamma$ (ordered lexicographically from the right), with respect to which $F$ has residue field $\mathbb{F}_q$. Let $a\in F$, $\g\in\Gamma$, $j\in\mathbb{Z}$; then
\[\int^F\psi_a(x)\Char{t(\g)\rho^{-1}(\pi^j\roi{\Res{F}})}(x)\,dx
    =\begin{cases}
    0   &\quad\g<-\nu(a)\\
    \int_{\pi^j\roi{\Res{F}}}{\Res{\psi}}(\eta(a)y)\,dyX^{\g}
&\quad\g=-\nu(a)\\
    \int^F\Char{t(\g)\rho^{-1}(\pi^j\roi{\Res{F}})}(x)\,dx&\quad\g>-\nu(a)\\
     \end{cases}
\]

Suppose further, for simplicity, that $\Res\psi$ is trivial on $\pi\roi{\Res{F}}$ but not on $\roi{\Res{F}}$, and that the Haar measure on $\res{F}$ has been chosen such that $\roi{\Res{F}}$ has measure $1$; then
\[\int_{\pi^j\roi{\Res{F}}}{\Res{\psi}}(\eta(a)y)\,dy
    =\begin{cases}
    0&\quad j\le-\nu_{\Res{F}}(\eta(a))\\
    q^{-j}\mu(\roi{\Res{F}})&\quad j>-\nu_{\Res{F}}(\eta(a))\\
     \end{cases}\]
Therefore
\[\int^F\psi_a(x)\Char{t(\g)\rho^{-1}(\pi^j\roi{\Res{F}})}(x)
    =\begin{cases}
    0&\quad w(a)<(-j+1,-\g)\\
    q^{-j}X^{\g}&\quad w(a)\ge(-j+1,-\g),\\
     \end{cases}\]

Finally, as $\Char{t(\g)\rho^{-1}(\pi^j\roi{\Res{F}}^{\times})}=
\Char{t(\g)\rho^{-1}(\pi^j\roi{\Res{F}})}
-\Char{t(\g)\rho^{-1}(\pi^{j+1}\roi{\Res{F}})}$, it follows that
\[\int^F\psi_a(x)\Char{t(\g)\rho^{-1}(\pi^j\roi{\Res{F}}^{\times})}(x)\,dx
    =\begin{cases}
    0&w(a)<(-j,-\g)\\
    q^{-j-1}X^{\g}&w(a)=(-j,-\g)\\
    q^{-j}(1-q^{-1})X^{\g}&w(a)>(-j,-\g)
    \end{cases}\]
\end{example}
%

\subsection{Harmonic Analysis}
Now that we can integrate products of functions and characters, we may define a Fourier transform on $F$:

\begin{definition}
Let $f$ be in $\calL(F,\psi)$. The Fourier transform of $f$, denoted $\hat{f}$, is the $\CG$-valued function on $F$ defined by $\hat{f}(x)=\int^F f(y)\psi(xy)\,dy$.
\end{definition}

The Fourier transforms on $F$ and $\res{F}$ are related as follows:
\begin{proposition}\label{prop_transform_of_lifted_function}
Let $g$ be Haar integrable on $F$, and $\g\in\Gamma$, $a,b\in F$; set $f=g^{a,\g}\psi_b$, the product of a lifted function with a good character. Then \[\hat{f}=\psi(ab)\hat{g}^{-b,-\g}\psi_a\,X^{\g}\] where $\hat{g}$ is the Fourier transform of $g$ with respect to $\Res{\psi}$.
\end{proposition}
\begin{proof}
By definition of the Fourier transform, $x\in F$ implies \[\hat{f}(x)=\int^F g^{a,\g}(y)\psi((b+x)y)\,dy. \tag{$\ast$}\] This is zero if $\g<-\nu(b+x)$ ie. if $x\notin\Coset{-b}{-\g}$. Conversely, suppose that $x=-b+t(-\g)x_0$, where $x_0\in\roi{F}$; then the integrand in ($\ast$) is \[g^{a,\g}\psi_{t(-\g)x_0}=\psi(t(-\g)ax_0)g^{a,\g}\res{\psi}_{\res{x}_0}^{a,\g},\] an identity which is easy checked by evaluating on $\Coset{a}{\g}$. So
\begin{align*}
\hat{f}(x)
	&=\psi(t(-\g)ay)\int^F g^{a,\g}(y)\res{\psi}_{\res{x}_0}^{a,\g}(y)\,dy\\
	&=\psi(t(-\g)ay)\hat{g}(\res{x}_0)\,X^{\g}\\
	&=\psi(a(x+b))\hat{g}(\res{x}_0)\,X^{\g},\end{align*}
which completes the proof.
\end{proof}

Let $\mathcal{S}(F,\psi)$ denote the subspace of $\calL(F,\psi)$ spanned over $\CG$ by functions of the form $g^{a,\g}\psi_b$, for $g$ a Schwartz-Bruhat function on $\res{F}$, $\g\in\Gamma$, $a,b\in F$. Recall that the Schwartz-Bruhat space on $\res{F}$ is invariant under the Fourier transform and that there exists a positive real $\lambda$ such that for any Schwartz-Bruhat function $g$, Fourier inversion holds: $\hat{\hat{g}}(x)=\lambda g(-x)$ for all $x\in\res{F}$. The following proposition extends these results to $F$: 

\begin{proposition}\label{prop_double_transform}
The space $\mathcal{S}(F,\psi)$ is invariant under the Fourier transform. For $f$ in $\calS(F,\psi)$, a double transform formula holds: $\hat{\hat{f}}(x)=\lambda f(-x)$ for all $x\in F$.
\end{proposition}
\begin{proof}
By linearity it suffices to consider the case $f=g^{a,\g}\psi_b$, for $\g\in\Gamma$, $a,b\in F$, and $g$ a Schwartz-Bruhat function on $\res{F}$. Then $\hat{f}=\psi(ab)\hat{g}^{-b,-\g}\psi_aX^{\g}$ belongs to $\mathcal{S}(F,\psi)$ and so \[\hat{\hat{f}}=\psi(ab)\,\psi(-ba)\hat{\hat{g}}^{-a,\g}\psi_{-b}X^{-\g}\,X^{g} =(\hat{\hat{g}})^{-a,\g}\psi_{-b},\] by proposition \ref{prop_transform_of_lifted_function}. Apply the inversion formula for $g$ to complete the proof.
\end{proof}

\begin{remark}
Let us consider the dependence of theory on the choice of character $\psi$; let $\psi'$ be another good character of $F$. In the interesting case of a higher local field, self-duality suggests that we may restrict attention to the case $\psi'=\psi_{\al}$ for some $\al\in\mult{F}$; so we assume henceforth that $\psi'=\psi_{\al}$. Then $\calL(F,\psi)=\calL(F,\psi')$, where $\calL(F,\psi')$ is defined in the same way as $\calL(F,\psi)$ but replacing $\psi$ by $\psi'$; further, the uniqueness of the extension of $\int^F$ given by proposition \ref{prop_extn_of_integral_to_characters} shows that the this extension does not depend on $\psi$.

Let $\frakf$ be the conductor of $\psi'$, and $\res{\psi'}$ the induced character of $\res{F}$; thus $\res{\psi'}(\res{x})=\psi'(t(\frakf-1)x)$ for $x\in\roi{F}$. By lemma \ref{lemma_translation_of_good_characters}, $\res{\psi'}=\res{\psi}_{\eta(\al)}$, and $\frakf=1-\nu(\al)$.

Let $g$ be Haar integrable on $F$, and $\g\in\Gamma$, $a,b\in F$; set $f=g^{a,\g}\psi_b$. Let $\check{f}$ denote the Fourier transform of $f$ with respect to $\psi'$; then for $y\in F$,
\begin{align*}\check{f}(y)
	&=\int^F f(x)\psi'(yx)\,dx\\
	&=\widehat{g^{a,\g}\psi_{\al b}}(\al y)\\
	&=\psi(\al a b)\hat{g}^{-\al b,-\g}(\al y)\psi_a(\al y)\,X^{\g}, \end{align*}
by proposition \ref{prop_transform_of_lifted_function}. Further, $y\mapsto\hat{g}^{-\al b,-\g}(\al y)$ is the lift of $v\mapsto\hat{g}(\eta(\al)v)$ at $-b,-\g-\nu(\al)$, an identity easily proved (or see the proof of lemma \ref{lemma_abs_value}). Also, $\hat{g}(\eta(\al)v)=\check{g}(v)$, where $\check{g}$ is the Fourier transform of $g$ with respect to $\res{\psi}'$, and so the analogue of proposition \ref{prop_transform_of_lifted_function} follows: \[\check{f}=\psi'(ab)\check{g}^{-b,-\g-\nu(\al)}\psi_a'\,X^{\g}.\]

For $f$ in $\calS(F,\psi')=\calS(F,\psi)$, the analogue of proposition \ref{prop_double_transform} now follows: $\check{\check{f}}=\check{\check{g}}^{-a,\g}\psi'_{-b}\,X^{-\nu(\al)}$. That is, \[\check{\check{f}}(x)=\lambda' f(-x)\,X^{\frakf-1}\] for all $x\in F$, where $\lambda'$ is the double transform constant associated to $\res{\psi'}$ (see the paragraph preceding proposition \ref{prop_double_transform}).

\end{remark}
%
%
%

\section{Integration on $\mult{F}$}

In this section, we consider integration over the multiplicative group $\mult{F}$. By analogy with the case of a local field, we are interested in those functions $\phi$ of $\mult{F}$ for which $x\mapsto \phi(x)\abs{x}^{-1}$ is integrable on $F$, where $\abs{\cdot}$ is a certain modulus defined below.

\subsection{Integration on $\mult{F}$}
Let $\abs{\cdot}=\abs{\cdot}_{\res{F}}$ denote the absolute value on $\res{F}$ normalised by the condition $\int g(\al x)\,dx=\abs{\al}^{-1}\int g(x)\,dx$ for $g\in\calL$, $\al\in\mult{F}$. First we lift this absolute value to $F$:

\begin{lemma}\label{lemma_abs_value}
Let $f$ be a $\CG$-valued integrable function on $F$ and $\alpha\in F^{\times}$. Then the scaled function $x\mapsto f(\alpha x)$ also belongs to $\calL(F)$, and \[\int^F f(\al x)\,dx=\abs{\eta(\al)}^{-1}X^{-\nu(\al)}\int^Ff(x)\,dx.\]
\end{lemma}
\begin{proof}
By linearity we may assume that $f$ is the lift of a function from $\calL$; $f=g^{a,\g}$ say. Then for all $x\in \al^{-1}(\Coset
a{\g})$,
\[
    f(\al x)=g(\Res{(\al x-a)t(-\g)})\\
        =g(\Res{\eta(\al)}\;\Res{(x-\al^{-1}a)t(\nu(\al)-\g)})\]
So the function $x\mapsto f(\al x)$ is the lift of the function $y\mapsto g(\Res{\eta(\al)}y)$ at $\al^{-1}a,\g-\nu(\al)$. This has integral
\begin{align*}\int_{\res{F}}g(\eta(\al)y)\,dy\;X^{\g-\nu(\al)}
    &=\abs{\eta(\al)}^{-1}\int_{\res{F}}g(y)\,dy\,X^{\g}X^{-\nu(\al)}\\
    &=\abs{\eta(\al)}^{-1}X^{-\nu(\al)}\int^F f(x)\,dx,\end{align*}
as required.
\end{proof}

\begin{remark}
Lemma \ref{lemma_abs_value} remains valid if $\calL(F)$ is replaced by $\calL(F,\psi)$.
\end{remark}

The lemma and remark suggest the follows definition:

\begin{definition}
Let $\al$ be in $\mult{F}$; the \emph{absolute value} of $\al$ is $\abs{\al}=\abs{\eta(\al)}X^{\nu(\al)}$.

Let $\calL(\mult{F},\psi)$ be the set of $\CG$-valued functions $\phi$ on $\mult{F}$ for which $x\mapsto \phi(x)\abs{x}^{-1}$, a function of $\mult{F}$, may be extended to $F$ to give a function in $\calL(F,\psi)$. The integral of such a function over $\mult{F}$ is defined to be \[\int^{\mult{F}}\phi(x)\,\dmult{x}=\int^F \phi(x)\abs{x}^{-1}\,dx,\] where the integrand on the right is really the extension of the function to $F$.
\end{definition}

\begin{remark}
There is no ambiguity in the definition of the integral over $\mult{F}$, for $x\mapsto\phi(x)\abs{x}^{-1}$ can have at most one extension to $\calL(F,\psi)$. This follows from the fact that $\calL(F,\psi)$ does not contain $\Char{\{0\}}$.
\end{remark}

$\calL(\mult{F},\psi)$ is a $\CG$-space of $\CG$-valued functions, and $\int^{\mult{F}}$ is a $\CG$-linear functional. Moreover, the integral is invariant under multiplication in the following sense:

\begin{proposition}
If $\phi$ belongs to $\calL(\mult{F},\psi)$ and $\al$ is in $\mult{F}$, then $x\mapsto\phi(\al x)$ belongs to $\calL(\mult{F},\phi)$ and $\int^{\mult{F}}\phi(\al x)\,d\overset{\mbox{\tiny$\times$}}{x}= \int^{\mult{F}}\phi(x)\,d\overset{\mbox{\tiny$\times$}}{x}$.
\end{proposition}
\begin{proof}
Let $x\mapsto \phi(x)\abs{x}^{-1}$ be the restriction to $\mult{F}$ of $f\in\calL(F,\psi)$, say. Then $x\mapsto \phi(\al x)\abs{x}^{-1}=\abs{\al} \phi(\al x)\abs{\al x}^{-1}$ is the restriction to $\mult{F}$ of $x\mapsto \abs{\al}f(\al x)$, which belongs to $\calL(F,\psi)$ by lemma \ref{lemma_abs_value}. By the same lemma,
\begin{align*}\int^{\mult{F}}\phi(\al x)\,\dmult{x}
	&=\int^F\abs{\al}f(\al x)\abs{x}^{-1}\,dx\\
	&=\abs{\al}\abs{\al}^{-1}\int^F f(x)\,dx\\
	&=\int^{\mult{F}}\phi(x)\,\dmult{x},\end{align*}
as required.
\end{proof}

\begin{example}\label{example_integral_over_mult_F}
Let us evaluate a couple of integrals on $\mult{F}$.
\begin{enumerate}
\item Let $g$ be Haar integrable on $\res{F}$, $a\in F$, $\g\in\Gamma$, and assume $0\notin a+t(\g)\roi{F}$. Let $\phi$ be the restriction of $g^{a,\g}$ to $\mult{F}$. Then $\phi\in\calL(\mult{F},\psi)$, and \[\int^{\mult{F}} \phi(x)\,\dmult{x}=\abs{a}^{-1}\int^F g^{a,\g}(x)\,dx.\]
Indeed, $x\in a+t(\g)\roi{F}$ implies $\eta(x)=\eta(a)$, and so $x\mapsto\phi(x)\abs{x}^{-1}$ is the restriction of $\abs{a}^{-1} g^{a,\g}$ to $\mult{F}$.

\item Let $g$ be Haar integrable on $\mult{\res{F}}$, and let $\phi$ be the function on $\mult{F}$ which vanishes off $\mult{\roi{F}}$ and satisfies $\phi(x)=g(\res{x})$ for $x\in\roi{F}$. Then $\phi\in\calL(\mult{F},\psi)$ and \[\int^{\mult{F}}\phi(x)\,\dmult{x}=\int g(x)\abs{x}^{-1}\,dx.\]

Indeed, let $h$ be the extension of $x\mapsto g(x)\abs{x}^{-1}$ to $F$ defined by $h(0)=0$. Then $h$ is Haar integrable on $F$, and $h^{0,0}\in\calL(F)$ restricts to the function of $\mult{F}$ given by $x\mapsto\phi(x)\abs{x}^{-1}$.
\end{enumerate}
\end{example}
%
%
%
%
%

\section{Local zeta integrals}\label{section_local_zeta_integrals}
In the remainder of the paper we will discuss (generalisations of) local zeta integrals. We now state the main results of local zeta integrals for the local field $\res{F}$; see \cite[chapter I.2]{Moreno}. Let $g$ be a Schwartz-Bruhat function on $\res{F}$, $\w$ a quasi-character of $\mult{\res{F}}$, and $s$ complex. The associated local zeta integral on $\res{F}$ is \[\zeta_{\res{F}}(g,\w,s)=\int_{\mult{\res{F}}} g(x)\w(x)\abs{x}^s\,\dmult{x};\] this is well-defined (ie. the integrand is integrable) for $\realpart{s}$ sufficiently large. Associated to the character $\w$ there is a meromorphic function $L(\w,s)$, the local L-function, with the following properties:
\begin{enumerate}
\item[(AC)] Analytic continuation, with the poles 'bounded' by the L-function: for all Schwartz-Bruhat functions $g$, $\zeta_{\res{F}}(g,\w,s)/L(\w,s)$, which initially only defines a holomorphic function for $\realpart{s}$ sufficiently large, in fact has analytic continuation to an entire function \[Z_{\res{F}}(g,\w,s)\] of $s$.
\item[(L)] 'Minimality' of the L-function: there is a Schwartz-Bruhat function $g$ for which \[Z_{\res{F}}(g,\w,s)=1\] for all $s$.
\item[(FE)] Functional equation: there is an entire function $\ep(\w,s)$, such that for all Schwartz-Bruhat functions $g$, \[Z_{\res{F}}(\hat{g},\w^{-1},1-s)=\varepsilon(\w,s)Z_{\res{F}}(g,\w,s).\] Moreover, $\ep(\chi,s)$ is of exponential type; that is $\ep(\chi,s)=aq^{bs}$ for some complex $a$ and integer $b$.
\end{enumerate}

Having lifted aspects of additive measure, multiplicative measure, and harmonic analysis from the local field $\res{F}$ up to $F$, we now turn to lifting these results for local zeta integrals. Later, in section \ref{section_two_dim_zeta_integrals}, we will assume that $F$ is a two-dimensional local field and consider a different, more arithmetic, local zeta integral. To avoid confusion between the two we may later refer to those in this section as being $\emph{one-dimensional}$; the terminology is justified by the fact that this section concerns lifting the usual (one-dimensional) zeta integrals on $\res{F}$ up to $F$.

\begin{definition}
For $f$ in $\calS(F,\psi)$, $\w:\mult{\roi{F}}\to\mult{\Comp}$ a homomorphism, and $s$ complex, the associated \emph{(one-dimensional) local zeta integral} is \[\zeta_F^1(f,\w,s)=\int^{\mult{F}} f(x)\w(x)\abs{x}^s\Char{\mult{\roi{F}}}(x)\dmult{x},\]
assuming that the integrand is integrable on $\mult{F}$.
\end{definition}

\begin{remark}
The inclusion of a $\Char{\mult{\roi{F}}}$ is solely for the benefit of the later study of two-dimensional zeta integrals where integrals over $\mult{\roi{F}}$ are more natural.
\end{remark}

We will focus on the situation where $\w$ is trivial on $1+t(1)\roi{F}$; that is, there is a homomorphism $\res{\w}:\mult{\res{F}}\to\mult{\Comp}$ such that $\w(x)=\res{\w}(\res{x})$ for all $x\in\mult{\roi{F}}$. If this induced homomorphism $\res{\w}$ is actually a quasi-character (ie. if it is continuous), then we will say that $\w$ is a \emph{good (multiplicative) character}.

\subsection{Explicit calculations and analytic continuation}
We now perform explicit calculations to obtain formulae for local zeta integrals attached to a good character:
\begin{lemma}
Let $\w$ be a good character of $\mult{\roi{F}}$; let $f=g^{a,\g}\psi_b$ be the product of a lifted function and a character, where $g$ is Schwartz-Bruhat on $\res{F}$, $a,b\in F$, $\g\in\Gamma$. Then we have explicit formulae for the local zeta integrals in the following cases:
\begin{enumerate}
\item Suppose that $\nu(a)<\min(\g,0)$; or that $0<\nu(a)<\g$; or that $0<\g\le\nu(a)$. Then $f(x)\w(x)\abs{x}^s\Char{\mult{\roi{F}}}(x)=0$ for all $x\in F$, $s\in\Comp$.
\item Suppose $0=\nu(a)<\g$. Then $f(x)\w(x)\abs{x}^s\Char{\mult{\roi{F}}}(x)=\w(a)\abs{a}^sf(x)$ for all $x\in F$, $s\in\Comp$; the local zeta integral is well-defined for all $s$ and is given by \[\zeta_F^1(f,\w,s)=\w(a)\abs{a}^{s-1}\int^Ff(x)\,dx.\]
\item Suppose $0=\g\le\nu(a)$. Then the local zeta integral is well-defined for $\realpart{s}$ sufficiently large, and is given by
\[\zeta_F^1(f,\w,s)=\begin{cases}
	\zeta_{\res{F}}(g_1\res{\psi}_{\res{b}},\res{w},s) & \mbox{if } \nu(b)\ge0 \\
	0 & \mbox{if }\nu(b)<0
	\end{cases}\]
where $g_1$ is the Schwartz-Bruhat function on $\res{F}$ given by $g_1(u)=g(u-a)$.
\end{enumerate}
\end{lemma}
\begin{proof}
In any of the cases in (i), $f$ vanishes off $\mult{\roi{F}}$; so $f(x)\Char{\mult{\roi{F}}}(x)=0$ for all $x\in F$.

In case (ii), $a+t(\g)\roi{F}$ is contained in $\mult{\roi{F}}$, and $x\in a+t(\g)\roi{F}$ implies $\w(x)\abs{x}^s=\w(a)\abs{a}^{s}$; this implies that $f(x)\w(x)\abs{x}^s\Char{\mult{\roi{F}}}(x)=f(x)\w(a)\abs{a}^s\Char{\mult{\roi{F}}}(x)$ for all $x\in F$, $s\in\Comp$. Moreover, for all $x\in F$, these results again imply $f(x)\abs{x}^{-1}=f(x)\abs{a}^{-1}$; therefore $f$ is integrable over $\mult{F}$, with $\int^{\mult{F}}f(x)\,\dmult{x}=\int^Ff(x)\,dx.$

Finally we turn to case (iii). First note that $g^{a,\g}\w\abs{\cdot}^{s-1}\Char{\mult{\roi{F}}}$ is the lift of $g_1\res{\w}\abs{\cdot}^{s-1}\Char{\mult{\res{F}}}$ at $0,0$. Now, if $\realpart{s}$ is sufficiently large then the theory of local zeta integrals for $\res{F}$ implies that $g_1\res{\w}\abs{\cdot}^{s-1}\Char{\mult{\res{F}}}$ is integrable on $\res{F}$; thus $f\w\abs{\cdot}^{s-1}\Char{\mult{\roi{F}}}$ is the restriction to $\mult{F}$ of $(g_1\res{\w}\abs{\cdot}^{s-1}\Char{\mult{\res{F}}})^{0,0}\psi_b$, a function which belong to $\calL(F,\psi)$.

By definition of the integral on $\mult{F}$ it follows that (for $\realpart{s}$ sufficiently large) $f\w\abs{\cdot}^{s-1}\Char{\mult{\roi{F}}}$ belongs to $\calL(\mult{F},\psi)$, and
\begin{align*}
\int^{\mult{F}} f(x)\w(x)\abs{x}^s\Char{\mult{\roi{F}}}(x)\,\dmult{x}
	&=\int^F(g_1\res{\w}\abs{\cdot}^{s-1}\Char{\mult{\res{F}}})^{0,0}(x)\psi_b(x)dx\\
	&=\begin{cases}
		\int^F(g_1\res{\w}\abs{\cdot}^{s-1}\Char{\mult{\res{F}}})^{0,0}(x)\,dx & \mbox{if } \nu(b)>0 \\
		\int^F(g_1\res{\w}\abs{\cdot}^{s-1}\Char{\mult{\res{F}}}\res{\psi_b})^{0,0}(x)\,dx & \mbox{if } \nu(b)=0 \\
		0 & \mbox{if }\nu(b)<0
		\end{cases}\\
	&=\begin{cases}
		\int g_1(u-a)\res{\w}(u)\abs{u}^{s-1}\Char{\mult{\res{F}}}(u)\,du & \mbox{if } \nu(b)>0 \\
		\int g_1(u-a)\res{\w}(u)\abs{u}^{s-1}\Char{\mult{\res{F}}}(u)\res{\psi_b}(u)\,du & \mbox{if } \nu(b)=0 \\
		0 & \mbox{if }\nu(b)<0
		\end{cases}\\
	&=\begin{cases}
		\zeta(g_1,\res{w},s) & \mbox{if } \nu(b)>0 \\
		\zeta(g_1\res{\psi}_{\res{b}},\res{w},s) & \mbox{if } \nu(b)=0 \\
		0 & \mbox{if }\nu(b)<0
		\end{cases}
\end{align*}
as required.
\end{proof}

\begin{remark}\label{remark_gaussian_sum_case}
Let $\w$ and $f=g^{a,\g}\psi_b$ be as in the statement of the previous lemma. The lemma treats all possible relations between $\nu(a)$, $\g$, and $0$ with the exception of $\nu(a)\ge\g<0$. There are complications in this case: since $f\Char{\mult{\roi{F}}}=f(0)\Char{\mult{\roi{F}}}$, we wish to calculate \[\zeta_F^1(f,\w,s) =f(0)\int^{\mult{F}}\psi_b(x)\w(x)\abs{x}\Char{\mult{\roi{F}}}(x)\,\dmult{x}.\]

For example, if $\psi_b$ has conductor $1$ then \[\psi_b\w\abs{\cdot}^s\Char{\mult{\roi{F}}}=(\res{\psi_b}\res{w}\abs{\cdot}^s\Char{\mult{\res{F}}})^{0,0}\] and so the zeta integral is formally given by \[\zeta_F^1(f,\w,s)=f(0)\int_{\mult{\res{F}}}\res{\psi_b}(x)\res{\w}(x)\abs{x}^s\,\dmult{x}.\] If $\res{F}$ were finite then this would be a gaussian sum over a finite field, a standard ingredient of local zeta integrals; with $\res{F}$ a local field it is unclear how to interpret this but the following examples provide insight.
\end{remark}

\begin{example}\label{example_PV_non_archimedean}
Suppose $\res{F}$ is non-archimedean and consider the formal integral \[\int_{\mult{\res{F}}}\psi_K(x)\w(x)\,\dmult{x}\] with $\psi_K$ an additive character and $\w$ a multiplicative quasi-character with $\realpart{\w}>0$ (recall that this is defined by $\abs{\w(x)}=\abs{x}^{\mbox{\scriptsize Re}(\w)}$ for all $x$). If $n$ is a sufficiently small integer, then we have a true integral \[\int_{w^{-1}(n)}\psi_K(x)\w(x)\,\dmult{x}=0,\] where $w$ is the discrete valuation of $\res{F}$; so for $n$ sufficiently small the value of the integral \[\int_{\{x:w(x)\ge n\}}\psi_K(x)\w(x)\,\dmult{x}=0\] does not depend on $n$. It seems reasonable to adopt this value as the meaning of the expression $\int_{\mult{\res{F}}}\psi_K(x)\w(x)\,\dmult{x}$.
\end{example}

\begin{example}\label{example_PV_archimedean}
Suppose $\res{F}=\Real$ and we wish to understand the formal integral \[\int_0^{\infty}e^{2\pi ix}\,dx.\] Replacing $2\pi i$ by some complex $\lambda$ with $\realpart{\lambda}<0$ gives a true integral with value \[\int_0^{\infty}e^{\lambda x}\,dx=-1/\lambda.\] Similarly we have \[\int_{-\infty}^0e^{\lambda x}\,dx=1/\lambda\] for $\realpart{\lambda}>0$. This suggests that, formally, \[\int_{\Real} e^{2\pi ix}\,dx=\int_{-\infty}^0e^{2\pi ix}\,dx+\int_0^{\infty}e^{2\pi ix}\,dx=0\] and \[\int_{\Real}e^{2\pi ix}\mbox{sign}(x)\,dx==\int_{-\infty}^0e^{2\pi ix}\,dx-\int_0^{\infty}e^{2\pi ix}\,dx=i/\pi\] where $\mbox{sign}(x)$ is the sign ($\pm$) of $x$.

The first of these integrals is already taken into account by our measure theory: if $F=\Real((t))$ and $\psi$ is the character defined by $\psi(\sum_na_nt^n)=e^{2\pi i a_0}$ (see example \ref{example_character_on_laurent_series}), then $\psi\Char{\roi{F}}$ belongs to $\calL(F,\psi)$ and $\int^F\psi(x)\Char{\roi{F}}(x)\,dx=0$. But $\psi\Char{\roi{F}}$ is also the lift of $x\mapsto e^{2\pi xi}$ at $0,0$ so formally $\int^F\psi(x)\Char{\roi{F}}(x)\,dx=\int_{\Real} e^{2\pi ix}\,dx$.

Such manipulations of integrals are common in quantum field theory (see eg. \cite{Johnson-Lapidus}) and I am grateful to Dr. Jorma Louko for discussions in this subject. That such integrals appear here further suggests a possible relation between this theory and Feynman path integrals. More evidence for such relations maybe found in sections 16 and 18 of \cite{Fesenko-analysis-on-loop-spaces}.
\end{example}

Ignoring the complications caused by this difficult case we may now deduce the first main properties of some local zeta functions. The appendix explains what is meant by a $\CG$-valued holomorphic function.
\begin{proposition}
Let $\w$ be a good character of $\mult{\roi{F}}$, and let $f$ be in $\calS(F,\psi)$; assume that $f$ may be written as a finite sum of terms $f=\sum_i g_i^{a_i,\g_i}\psi_{b_i}\,p_i$ where each $g_i^{a_i,\g_i}\psi_{b_i}$ is treated by one of the cases of the previous lemma and $p_i\in\CG$. Then
\begin{enumerate}
\item For $\realpart{s}$ sufficiently large, the integrand of the local zeta integral $\zeta_F^1(f,\w,s)$ is integrable over $\mult{F}$ and so the local zeta integral is well-defined.
\item $\zeta^F(f,\w,s)/L(\res{\w},s)$ has entire analytic continuation: that is, there is a $\CG$-valued holomorphic function $Z_F^1(f,\w,s)$ on $\Comp$ which equals $\zeta_F^1(f,\w,s)/L(\res{\w},s)$ for $\realpart{s}$ sufficiently large.
\item There is some function $g\in\calS(F,\psi)$ for which $Z_F^1(g,\w,s)=1$ for all complex $s$.
\end{enumerate}
\end{proposition}
\begin{proof}
The results follow by linearity, the previous lemma, and the main properties of local zeta integrals on $\res{F}$.
\end{proof}

It is important to extend this result to all $f$ in $\calS(F,\psi)$; therefore the complication discussed in remark \ref{remark_gaussian_sum_case} must be resolved.

\begin{remark} We now briefly consider functional equations. These is no result as satisfactory as for zeta functions of a one-dimensional local field, and there is no reason why there should be due to the $\Char{\mult{\roi{F}}}$ factor. The most interesting issue here is making a functional equation compatible with the difficulties caused by remark \ref{remark_gaussian_sum_case}; this should indicate correctness (or not) of examples \ref{example_PV_non_archimedean} and \ref{example_PV_archimedean}.
\end{remark}
%
%
%

\section{Local functional equations with respect to $s$ goes to $2-s$}\label{section_s_goes_to_2-s}
In this section we continue our study of local zeta functions, considering the problem of modifying the functional equation (FE) on $\res{F}$ so that the symmetry is not $s$ goes to $1-s$, but instead $s$ goes to $2-s$. This is in anticipation of the next section on two-dimensional zeta integrals, where such a functional equation is natural.

Since this section is devoted to the residue field $\res{F}$, we write $K=\res{F}$. We fix an non-trivial additive character $\psi_K$ of $K$ (until proposition \ref{proposition_varying_prime_and_characters} where we consider dependence on this choice). Fourier transforms of complex-valued functions are taken with respect to this character (and the measure which was fixed at the start of the paper): $\hat{g}(y)=\int g(x)\psi_K(xy)\,dx$.

The two main proofs of (FE) are Tate's \cite{Tate's-thesis} using Fubini's theorem, and Weil's \cite{Weil-fonction-zeta} using distributions. For Weil, a fundamental identity in the non-archimedean case is
\[\widehat{g(\al\,\cdot)}=\abs{\al}^{-1}\hat{g}(\al^{-1}\,\cdot)\tag{$\ast$}\]
for $\al\in\mult{K}$, where we write $g(\al\,\cdot)$ for the function $x\mapsto g(\al x)$, notation which we shall continue to use.

The aim of this section is to replace the Fourier transform with a new transform so that ($\ast$) holds with $\abs{\al}^{-2}$ in place of $\abs{\al}^{-1}$. This leads to a modification of the local functional equation, with $\abs{\cdot}^2$ in place of $\abs{\cdot}$; see propositions \ref{proposition_*_endomorphism} and \ref{proposition_archimedean_functional_equation}.
%
%
\subsection{Non-archimedean case}
We assume first that $K$ is a non-archimedean local field. The following proposition precisely explains the importance of the identity $\widehat{g(\al\,\cdot)}=\abs{\al}^{-1}\hat{g}(\al^{-1}\,\cdot)$

\begin{proposition}\label{proposition_*_endomorphism}
Suppose that $g\mapsto g^*$ is a complex linear endomorphism of the Schwartz-Bruhat space $\calS(K)$ of $K$ which satisfies, for some integer $n$, \[g(\al\,\cdot)^*=\abs{\al}^{-n}g^*(\al^{-1}\,\cdot)\] for all $g\in\calS(K)$, $\al\in\mult{K}$. Let $\w$ be a quasi-character of $\mult{K}$. Then there is a unique entire function $\varepsilon_*(\w,s)$ which satisfies \[Z(g^*,\w^{-1},n-s)=\ep_*(\w,s)Z(g,\w,s)\] for all $g\in\calS(K)$, $\al\in\mult{K}$.
\end{proposition}
\begin{proof}
Let $g$ be a Schwartz-Bruhat function on $K$, and $\al\in\mult{K}$. Then for $\realpart{s}$ sufficiently large to ensure integrablility, the identity \[\zeta(g(\al\cdot),\w,s)=\w(\al)^{-1}\abs{\al}^{-s}\zeta(g,\w,s)\] holds. Conversely, for $\realpart{s}$ sufficiently small, the assumed property of $^*$ implies that \[\zeta(g(\al\cdot)^*,\w^{-1},n-s)=\w(\al)^{-1}\abs{\al}^{-s}\zeta(g^*,\w^{-1},n-s).\] Therefore, for all complex $s$, \[Z(g(\al\cdot),\w,s)=\w(\al)^{-1}\abs{\al}^{-s}Z(g,\w,s)\] and \[Z(g(\al\cdot)^*,\w^{-1},n-s)=\w(\al)^{-1}\abs{\al}^{-s}Z(g^*,\w^{-1},n-s).\] So the complex linear functionals $\Lambda$ on $\calS(K)$ given by \[g\mapsto Z(g,\w,s)\] and \[g\mapsto Z(g^*,\w^{-1},n-s)\] (for fixed $s$) each satisfy $\Lambda(g(\al\cdot))=\w(\al)^{-1}\abs{\al}^{-s}\Lambda(g)$ for all $g\in\calS(K)$, $\al\in\mult{K}$. But the space of such functionals is one-dimensional (see eg. \cite[chapter I.2]{Moreno}) (for $\w\neq\abs{\cdot}^{-s}$) and there is $f\in\calS(K)$ such that $Z(f,\w,s)=1$ for all $s$ (property (L) of local zeta integrals); this implies the existence of an entire function $\ep_*(\w,s)$ as required.
\end{proof}

\begin{remark}
Suppose that $^*$ maps $\calS(K)$ \emph{onto} $\calS(K)$. Then there is $g\in\calS(K)$ such that $Z(g^*,\w^{-1},1-s)=1$ for all $s$ and so $\ep_*(\w,s)$ is nowhere vanishing.
\end{remark}

Our aim now is to investigate the epsilon factors attached to a particular transform $^*$ which satisfies $g(\al\,\cdot)^*=\abs{\al}^{-2}g^*(\al^{-1}\,\cdot)$. Let $w:\mult{K}\to\mathbb{Z}$ be the discrete valuation of $K$ and $\pi$ a prime.
\begin{definition}
Define
\begin{align*}
    \nab:&K\to K\\
    &x\mapsto\pi^{w(x)}x
\end{align*}
(and $\nab(0)=0$).

For $g$ be a complex-valued function on $K$, denote by $Wg$ the function
\[Wg(x)=
	\begin{cases}
		g(\pi^{-w(x)/2}x) &\mbox{if $w(x)$ is even}\\
		g(\pi^{(-w(x)-1)/2}x) &\mbox{if $w(x)$ is odd}
	\end{cases}\]
(and $Wg(0)=g(0)$).

Assuming that $Wg$ is integrable on $K$, define the \emph{$^*$-transform} (with respect to $\pi$) of $g$ by \[g^*=\widehat{Wg}\circ\nab.\]
\end{definition}

\begin{remark}
Compare this definition with \cite{Weil-fonction-zeta} and [AoAS] section 15 where Fesenko defines the transform on two copies of a two-dimensional local field $F\times F$.

The $^*$-transform depends on choice of prime $\pi$. We may also denote by $\nab$ the composition operator $\nab(g)=g\circ\nab$.

The space of Schwartz-Bruhat functions $\calS(K)$ is closed under the $^*$-transform.
\end{remark}

It is easy to verify the the $^*$-transform has the desired property:
\begin{lemma} \label{lemma_*_transform_1}
Suppose that $g$ is a Schwartz-Bruhat on $K$ and that $\al\in\mult{K}$. Then \[g(\al\,\cdot)^*=\abs{\al}^{-2}g^*(\al^{-1}\,\cdot).\]
\end{lemma}
\begin{proof}
If $x\in\mult{F}$, then $W(g(\al\,\cdot))(x)=W(g)(\pi^{w(\al)}\al x)$. Hence \[\widehat{W(g(\al\cdot))}=\abs{\pi^{w(\al)}\al}^{-1}\widehat{Wg}(\pi^{-w(\al)}\al^{-1}\,\cdot).\] Evaluating this at $\nab(x)$ yields
\[g(\al\,\cdot)^*(x)
    =\abs{\al}^{-2}\widehat{Wg}(\pi^{-w(\al)}\al^{-1}\pi^{w(x)}x)
    =\abs{\al}^{-2}g^*(\al^{-1} x).\]
\end{proof}

\begin{remark}
More generally, the previous lemma holds for any complex valued $g$ for which $Wg$ and $W(g(\al\cdot))$ are both integrable.
\end{remark}

We now $^*$-transform several functions. Let $\mu$ be the measure of $\roi{K}$ under our chosen Haar measure and let $d$ be the conductor of $\psi_K$.
\begin{example} \label{eg_*_transforms1}
Suppose $g=\Char{\pi^r\roi{K}}$. Then $Wg=\Char{\pi^{2r}\roi{K}}$, which has Fourier transform $\mu q^{-2r}\Char{\pi^{d-2r}\roi{K}}$. So the $^*$-transform of $g$ is \[g^*=\mu q^{-2r}\Char{\pi^{\lceil{d/2}\rceil-r}\roi{K}},\] where $\lceil{d/2}\rceil$ denotes the least integer not strictly less than $d/2$. Compare this with the Fourier transform \[\hat{g}=\mu q^{-r}\Char{\pi^{d-r}\roi{K}}.\]
\end{example}

\begin{example} \label{eg_*_transforms2}
Suppose $h=\Char{1+\pi^r\roi{K}}$ with $r\ge 1$. Let $x\in\mult{K}$. If $w(x)$ is even, then $Wh(x)=1$ if and only if $x\in 1+\pi^r\roi{K}$; if $w(x)$ is even, then $Wh(x)=1$ if and only if $\pi^{-1}x\in a+\pi^r\roi{K}$. So \[Wh=\Char{1+\pi^r\roi{K}}+\Char{\pi(1+\pi^r\roi{K})},\] whence \[\widehat{Wh}=\mu q^{-r}\Char{\pi^{d-r}\roi{K}}\psi_K+\mu q^{-r-1}\Char{\pi^{d-r-1}\roi{K}}\psi_K(\pi\,\cdot).\]

For the remainder of this example assume $\mu=1$, $d=0$, $r=2$; we shall compute the double $^*$-transform $h^{**}$.

It may be easily checked that if $x\in K$, then
\[\Char{\pi^{-2}\roi{K}}(\nab(x))\psi_K(\nab(x))
	=\begin{cases}
	0			&\mbox{if }x\notin\pi^{-1}\roi{K}\\
	\psi_K(\pi^{-1}x)	&\mbox{if }x\in\pi^{-1}\mult{\roi{K}}\\
	1			&\mbox{if }x\in\roi{K}
	\end{cases}\]
and
\[\Char{\pi^{-3}\roi{K}}(\nab(x))\psi_K(\pi\nab(x))
	=\begin{cases}
	0			&\mbox{if }x\notin\pi^{-1}\roi{K}\\
	\psi_K(x)		&\mbox{if }x\in\pi^{-1}\mult{\roi{K}}\\
	1			&\mbox{if }x\in\roi{K}.
	\end{cases}\]
From the identity for $\widehat{Wh}$ it now follows that \[h^*=q^{-2}(\psi_K(\pi^{-1}\,\cdot)+q^{-1}\psi_K)\,\Char{\pi^{-1}\mult{\roi{K}}}+q^{-2}(1+q^{-1})\Char{\roi{K}}.\] Set $h_1=\psi_K(\pi^{-1}\,\cdot)\,\Char{\pi^{-1}\mult{\roi{K}}}$, $h_2=q^{-1}\psi_K\,\Char{\pi^{-1}\mult{\roi{K}}}$; it may be checked that
\begin{align*} 
	Wh_1&=\psi_K(\pi^{-1}\,\cdot)\,\Char{\pi^{-1}\mult{\roi{K}}}+\psi_K\,\Char{\pi^{-2}\mult{\roi{K}}}\\
	Wh_2&=q^{-1}\psi_K\,\Char{\pi^{-1}\mult{\roi{K}}}+q^{-1}\psi_K(\pi\,\cdot)\,\Char{\pi^{-2}\mult{\roi{K}}}
	\end{align*}
Standard Fourier transform calculations now yield
\begin{align*}
	\widehat{Wh_1}&=q\,\Char{-\pi^{-1}+\pi\roi{K}}-\Char{-\pi^{-1}+\roi{K}}+q^2\,\Char{-1+\pi^2\roi{K}}-q\,\Char{-1+\pi\roi{K}}\\
	\widehat{Wh_2}&=\Char{-1+\pi\roi{K}}-q^{-1}\,\Char{\roi{K}}+q\,\Char{-\pi^{-1}+\pi^2\roi{K}}-\Char{\pi\roi{K}}.
\end{align*}
Further, by example \ref{eg_*_transforms1}, $\widehat{W(\Char{\roi{K}})}=\Char{\roi{K}}$, and so
\begin{align*}
	q^2\widehat{W(h^*)}
		&=q\,\Char{-\pi^{-1}+\pi\roi{K}}-\Char{-\pi^{-1}+\roi{K}}+q^2\,\Char{-1+\pi^2\roi{K}}-q\,\Char{-1+\pi\roi{K}}\\
		&\qquad+\Char{-1+\pi\roi{K}}+q\,\Char{-\pi^{-1}+\pi^2\roi{K}}-\Char{\pi\roi{K}}+\Char{\roi{K}}.
\end{align*}
Now, $x\in\mult{K}$ implies $w(\nab x)$ is even, and so
\begin{align*}
	q^2\widehat{W(h^*)}\circ\nab
		&=q^2\,\Char{-1+\pi^2\roi{K}}\circ\nab-q\,\Char{-1+\pi\roi{K}}\circ\nab\\
		&\qquad+\Char{-1+\pi\roi{K}}\circ\nab-\Char{\pi\roi{K}}\circ\nab+\Char{\roi{K}}\circ\nab\\
		&=q^2\,\Char{-1+\pi^2\roi{K}}-q\,\Char{-1+\pi\roi{K}}\\
		&\qquad+\Char{-1+\pi\roi{K}}-\Char{\pi\roi{K}}+\Char{\roi{K}}.
\end{align*}
That is,
\[h^{**}=q^{-2}\,\Char{\mult{\roi{K}}}-q^{-1}(1-q^{-1})\Char{-1+\pi\roi{K}}+\Char{-1+\pi^2\roi{K}}.\]

Note that although the definition of the $^*$-transform depends on choice of prime $\pi$, the double $^*$-transform $h^{**}$ of $h$ does not. This will be proved in general below.
\end{example}

Let us now obtain explicit formulae for the epsilon factors $\ep_*(\w,s)$.

\begin{example}
We calculate the epsilon factor attached to the $^*$-transform for the trivial character $1$. Suppose for simplicity that $\roi{K}$ has measure $1$ under our chosen Haar measure.

Let $f=\Char{\roi{K}}$. Example \ref{eg_*_transforms1} implies $f^*=\Char{\pi^{\lceil{d/2}\rceil-r}\roi{K}}$; it is a standard calculation that $Z(f,1,s)=1-q^{-1}$ and $Z(f^*,1,2-s)=(1-q^{-1})q^{\lceil{d/2}\rceil(s-2)}$ for all $s$. Therefore \[\ep_*(1,s)=q^{\lceil{d/2}\rceil(s-2)}\] for all $s$.
\end{example}

\begin{example}
We now calculate the epsilon factor attached to the $^*$-transform for ramified quasi-characters. Continue to suppose that that $\roi{K}$ has measure $1$.

Let $\w$ be a quasi-character of $\mult{K}$ of conductor $r>0$; that is, $\w|_{1+\pi^r\roi{K}}=1$ but $\w|_{1+\pi^{r-1}\roi{K}}\not=1$

Let $h=\Char{1+\pi^r\roi{K}}$; so $\zeta(h,\w,s)$ is constantly $m$, the measure of $1+\pi^r\roi{K}$ under $\dmult{x}=\abs{x}^{-1}dx$. The aim is now to calculate $\zeta(h^*,\w^{-1},2-s)$ without calculating $h^{*}$. By example \ref{eg_*_transforms2}, $Wh=h+h(\pi^{-1}\cdot)$, and so $\widehat{Wh}=\hat{h}+q^{-1}\hat{h}(\pi\cdot)$. Therefore
\begin{align*}
\zeta(h^*,\w^{-1},2-s)
	&=\int_{\mult{K}} \hat{h}(\pi^{\w(x)}x)\w(x)^{-1}\abs{x}^{2-s}\dmult{x}\\
		&\qquad\qquad\qquad\qquad+q^{-1}\int_{\mult{K}} \hat{h}(\pi^{\w(x)+1}x)\w(x)^{-1}\abs{x}^{2-s}\dmult{x}\\
	&=\sum_{n\in\mathbb{Z}}q^{n(s-2)}\int_{\w^{-1}(n)}\hat{h}(\pi^{n}x)\w(x)^{-1}\dmult{x}\\
		&\qquad\qquad\qquad\qquad+q^{-1}\sum_{n\in\mathbb{Z}}q^{n(s-2)}\int_{\w^{-1}(n)}\hat{h}(\pi^{n+1}x)\w(x)^{-1}\dmult{x}\\
	&=\sum_n q^{n(s-2)}\w(\pi)^{-n}\int_{\mult{\roi{K}}}\hat{h}(\pi^{2n}x)\w(x)^{-1}\dmult{x}\\
		&\qquad\qquad\qquad\qquad+q^{-1}\sum_nq^{n(s-2)}\w(\pi)^{-n-1}\int_{\mult{\roi{K}}}\hat{h}(\pi^{2n+1}x)\w(x)^{-1}\dmult{x}
\end{align*}

But by Tate's calculation \cite{Tate's-thesis} when calculating the epsilon factor in this same case,
\[\int_{\mult{\roi{K}}}\hat{h}(\pi^{N}x)\w(x)^{-1}\dmult{x}=\begin{cases}q^{-r/2}m\rho_0(\w^{-1}) & \mbox{if }N=d-r\\0&\mbox{otherwise}\end{cases},\] where $\rho_0(\w^{-1})$ is the \emph{root number} of absolute value one \[\rho_0(\w^{-1})=q^{-r/2}\sum_{\theta}\w^{-1}(\theta)\psi_K(\pi^{d-r}\theta),\] the sum being taken over coset representatives of $1+\pi^r\roi{K}$ in $\mult{\roi{K}}$.

Therefore
\begin{align*}
	\zeta(h^*,\w^{-1},2-s)
	&=\begin{cases}
		q^{(d-r)(s-2)/2}\w(\pi)^{(r-d)/2)}q^{-r/2}m\rho_0(\w^{-1})& d-r\mbox{ even} \\ 
		q^{(d-r-1)(s-2)/2-1}\w(\pi)^{(1+r-d)/2)}q^{-r/2}m\rho_0(\w^{-1}) & d-r\mbox{ odd}
		\end{cases}\\
	&=q^{\lceil{(r-d)/2}\rceil(2-s)}\w(\pi)^{\lceil{(r-d)/2}\rceil}q^{-r/2}\delta_{d-r} m\rho_0(\w^{-1})
\end{align*}
where $\delta_{d-r}=1$ if $r-d$ is even and $=q^{-1}$ if $r-d$ is odd. Finally, as we have already observed that $\zeta(h,\w,s)=m$ for all $s$, and $L(\w,s)=1$ for such a character, \[\ep_*(\w,s)=q^{\lceil{(r-d)/2}\rceil(2-s)}\w(\pi)^{\lceil{(r-d)/2}\rceil}q^{-r/2}\delta_{d-r}\rho_0(\w^{-1}).\]
\end{example}

\begin{remark}\label{remark_double_epsilon}
More generally, if $\roi{K}$ has measure $\mu$ under our chosen Haar measure, then each of the epsilon factors above is multiplied by a factor of $\mu$.

Let us now consider what happens when we take the double transform $f^{**}$. If $\w$ is ramified with conductor $r$, then
\begin{align*}
	\ep_*(\w,s)\ep_*(\w^{-1},2-s)
	&=\mu^2q^{2\lceil{(r-d)/2}\rceil}\delta_{d-r}^2q^{-r}\rho_0(\w^{-1})\rho_0(\w)\\
	&=\mu^2q^{2\lceil{(r-d)/2}\rceil}\delta_{d-r}^2q^{-r}\w(-1)\overline{\rho_0(\w)}\rho_0(\w)\\
	&=\mu^2q^{r-d}\delta_{d-r} q^{-r}\w(-1)\\
	&=\mu^2q^{-d}\delta_{d-r}\w(-1).
\end{align*}
If we declare the conductor of an unramified character to be $0$ then this formula remains valid for unramified $\w$.

Therefore two applications of the functional equation imply that for all $f\in\calS(K)$, all characters $\w$ of conductor $r\ge 0$, and all complex $s$, \[\zeta(f^{**},\w,s)=\mu^2q^{-d}\delta_{d-r}\w(-1)\zeta(f,\w,s).\tag{A}\]
\end{remark}

We will now proceed to use our results on epsilon factors to deduce properties of the $^*$-transform; the idea is to use identities between zeta integrals to obtain identities between the functions. The following result is clearly of great importance in this method:

\begin{lemma}\label{lemma_zeta_integrals_determine_fn}
Let $f\in\calS(K)$ and suppose that $\zeta(f,\w,s)=0$ for all quasi-characters $\w$ and complex $s$; then $f=0$.
\end{lemma}
\begin{proof}
Let $f$ be in $\calS(K)$. Then $f-f(0)\Char{\roi{K}}$ belongs to $\calS(\mult{K})$ and so the zeta integral $\zeta(f(0)\Char{\roi{K}},\w,s)$ is well-defined for all $s$ and belongs to $\Comp[q^{s},q^{-s}]$. Indeed, it suffices to observe that $\calS(\mult{K})$ is spanned by $\Char{a+\pi^m\roi{K}}$ where $w(a)>m$, and $\zeta(\Char{a+\pi^m\roi{K}},\w,s)=q^{-w(a)s}\int_{a+\pi^m\roi{K}}\w(s)\,\dmult{x}$.

However, for $\w=1$ the trivial character, \[\zeta(f(0)\Char{\roi{K}},1,s)=f(0)m(1-q^{-s})^{-1}\] where $m$ is the multiplicative measure of $\mult{\roi{K}}$. So the assumption that $\zeta(f,1,s)=0$ implies $f(0)(1-q^{-s})^{-1}\in\Comp[q^s,q^{-s}]$ as a function of $s$. This is false unless $f(0)=0$; therefore $f(0)=0$ and $f\in\calS(\mult{F})$.

So now $\zeta(f,\w,1)$ is well-defined for all characters $\w$ of $\mult{F}$ and equals $\tilde{f}(\w)$ where $\tilde{}$ denotes Fourier transform on the group $\mult{K}$ - so $\tilde{f}$ is a function on the dual group of $X(\mult{K})$ of $\mult{K}$. By the injectivity of the Fourier transform (see eg. \cite[chapter IV]{Gelfand}) from $L^1(\mult{K})$ to $C(X(\mult{K}))$ our hypothesis implies that $f=0$.
\end{proof}

We will now use identity (A) to prove results about the $^*$-transform. Recall that the transform depends on the choice of both non-trivial additive character and prime; surprisingly, the double $^*$-transform does not depend on choice of prime:

\begin{proposition}\label{proposition_varying_prime_and_characters}
The double $^*$-transform does not depend on choice of prime $\pi$. If the character $\psi_K$ is replaced by some other character, with conductor $d'$ say, and we assume that $d'=d\mod{2}$, then the double $^*$-transform is multiplied by a constant factor of $q^{d'-d}$.
\end{proposition}
\begin{proof}
Write more generally $D_i$ for the double $^*$-transform with respect to prime $\pi_i$ and character $\psi_K^i$ for $i=1,2$; let $d_i$ be the conductor of $\psi_K^i$ - assume $d_1=d_2\mod{2}$. Equation (A) implies that for all $f\in\calS(K)$, all characters $\w$ of conductor $r\ge 0$, and all complex $s$
\begin{align*}
	\zeta(D_1f,\w,s)
	&=\mu^2q^{-d_1}\delta_{d_1-r}\w(-1)\zeta(f,\w,s)\\
	&=q^{d_2-d_1}\delta_{d_2-r}\zeta(f,\w,s).
\end{align*}

Lemma \ref{lemma_zeta_integrals_determine_fn} implies now that $D_1f=q^{d_2-d_1}D_2f$, revealing the independence from the prime and claimed dependence on the conductor of the character.
\end{proof}

We will now use the identity (A) to prove that $^*$ is an automorphism of $\calS(K)$. It is interesting that we are using properties of zeta integrals and epsilon factors to deduce properties of $^*$; one would usually work in the other direction but the author could find no direct proof and it is very satisfying to apply zeta integrals to such a problem!

\begin{proposition}
The $^*$-transform is an linear automorphism of $\calS(K)$.
\end{proposition}
\begin{proof}
Let $D$ denote the double $^*$-transform on $\calS(K)$ with respect to our chosen character (we have shown that it does not depend on choice of prime); let $D_1$ denote the double $^*$-transform on $\calS(K)$ with respect to a character $\psi_K^1$ with conductor $d_1\neq d\mod{2}$. Equation (A) implies that for all $f\in\calS(K)$, all characters $\w$ of conductor $r\ge 0$, and all complex $s$
\begin{align*}
	\zeta(D_1Df,\w,s)
	&=\mu^2q^{-d_1}\delta_{d_1-r}\w(-1)\zeta(Df,\w,s)\\
	&=\mu^4q^{-d-d_1}\delta_{d-r}\delta_{d_1-r}\w(-1)^2\zeta(f,\w,s)\\
	&=\mu^4q^{-d-d_1}q^{-1}\zeta(f,\w,s)
\end{align*}
as $\delta_{d-r}\delta_{d_1-r}=q^{-1}$ for all $r$.

Lemma \ref{lemma_zeta_integrals_determine_fn} now implies that $D_1Df=\mu^4q^{-d-d_1}q^{-1}f$ for all $f\in\calS(K)$. Therefore $^*$ is injective. Replacing $D_1D$ by $DD_1$ in the argument similarly shows that $^*$ is surjective.
\end{proof}

\begin{remark}
The key to the previous proof is the identity $\delta_{d-r}\delta_{d_1-r}=q^{-1}$, which removes the dependence on the conductor $r$ of the multiplicative character. There is no clear way to relate zeta integrals of $f^{**}$ with those of $f$ in a manner independent of the character; so we were forced to transform four times!
\end{remark}

The following result shows that if $\psi_K$ has conductor $0$ then the $^*$-transform and Fourier transform agree on functions lifted from the residue field $\res{K}$:

\begin{proposition}\label{prop_*=Fourier}
Assume that the conductor of $\psi_K$ is $0$. Let $h$ be a complex-valued function on $\res{K}$ and $r$ an integer; let $f=h^{0,r}$ be the lift of $h$ at $0,r$ (that is, $f$ vanishes off $\pi^r\roi{K}$ and satisfies $f(\pi^r x)=h(\res{x})$ for $x\in\roi{K}$). Then $f^*=q^{-r-1}\hat{f}$.
\end{proposition}
\begin{proof}
Suppose initially that $r=-1$; to prove the assertion it suffices to consider functions $f=\Char{a+\roi{K}}$ for $a\in\pi^{-1}\roi{K}$. For such an $f$ it is easily checked that $W'(f)=f$ and $f^*=\hat{f}$.

For arbitrary $r$, note that $x\mapsto f(\pi^{r+1}x)$ satisfies the hypotheses for the $r=-1$ case; lemma \ref{lemma_*_transform_1} and the corresponding result for the Fourier transform, namely $f(\al\,\cdot)^*=\abs{\al}^{-1}\hat{f}(\al\,\cdot)$ for $\al\in\mult{K}$, imply $f^*=q^{-r-1}\hat{f}$.
\end{proof}

Let us summarise the main results of this section concerning local zeta integrals, the $^*$-transform, and related epsilon factors.

\begin{proposition}
Let $\w$ be a quasi-character of $\mult{K}$. Then
\begin{enumerate}
\item[(AC*)] Analytic continuation, with the poles 'bounded' by the L-function: for all Schwartz-Bruhat functions $g$, $\zeta_{\res{F}}(g,\w,s)/L(\w,s)$, which initially only defines a holomorphic function for $\realpart{s}$ sufficiently large, in fact has analytic continuation to an entire function \[Z_{\res{F}}(g,\w,s)\] of $s$.
\item[(L*)] 'Minimality' of the L-function: there is a Schwartz-Bruhat function $g$ for which \[Z_{\res{F}}(g,\w,s)=1\] for all $s$.
\item[(FE*)] Functional equation: there is an entire function $\ep_*(\w,s)$, such that for all Schwartz-Bruhat functions $g$, \[Z_{\res{F}}(\hat{g},\w^{-1},1-s)=\varepsilon(\w,s)Z_{\res{F}}(g,\w,s).\] Moreover, $\ep_*(\w,s)$ is of exponential type; that is, $\ep_*(\w,s)=aq^{bs}$ for some complex $a$ and integer $b$.
\end{enumerate}
\end{proposition}
\begin{proof}
Properties (AC*) and (L*) are just (AC) and (L) because they are independent of the chosen transform. (FE*) is proposition \ref{proposition_*_endomorphism} and the epsilon factors were shown to be of exponential type by explicit calculation in the examples.
\end{proof}

\begin{remark}\label{remark_non_arch_product_*_functional_eqn}
For applications to zeta-integrals on two-dimensional local fields we will require the $^*$-transform and zeta integrals for functions defined on the product space $K\times K$. As $\calS(K\times K)=\calS(K)\otimes\calS(K)$, we may just define the $^*$-transform on $\calS(K\times K)$ by $(f\otimes g)^*=f^*\otimes g^*$ and linearity.

Suppose that $\w$ is a quasi-character of $\mult{K}\times\mult{K}$; write $\w(x,y)=\w_1(x)\w_2(y)$ for quasi-characters $\w_i$ of $\mult{K}$. The decomposition $\calS(K\times K)=\calS(K)\otimes\calS(K)$ and previous proposition imply
\begin{enumerate}
\item For all $f\in\calS(K\times K)$, the integral $\zeta_{K\times K}(f,\w,s)=\int\int f(x,y)\w(x,y)\abs{x}^s\abs{y}^s\,\dmult{x}\dmult{y}$ is well-defined for $\realpart{s}$ large enough. Moreover, $s\mapsto\zeta_{K\times K}(f,\w,s)/(L(\w_1,s)L(\w_2,s))$ has analytic continuation to an entire function $Z_{K\times K}(f,\w,s)$.
\item There is $f\in\calS(K\times K)$ such that $Z_{K\times K}(f,\w,s)=1$ for all $s$.
\item For all $f\in\calS(K\times K)$, there is a functional equation: \[Z_{K\times K}(f^*,\w^{-1},2-s)=\ep_*(\w_1,s)\ep_*(\w_2,s)Z_{K\times K}(f,\w,s)\] for all $s$. Note that $\ep_*(\w_1,s)\ep_*(\w_2,s)$ is of exponential type.
\end{enumerate}
\end{remark}
%
%

\subsection{Archimedean case} \label{subsection_arch_functional_eqn}
Now suppose that $K$ is an archimedean local field. Rather than present a version of proposition \ref{proposition_*_endomorphism} using tempered distributions, we will just define an analogue of the $^*$-transform considered above and investigate its properties. The existence of an $s$ goes to $2-s$ functional equation will shown as in \cite{Tate's-thesis}, via Fubini's theorem.

\begin{definition}
Introduce
\begin{align*}
    \nab:&K\to K\\
    &x\mapsto\abs{x}x.
\end{align*}
Note that this $\nab$ is a bijection with inverse $x\mapsto x\abs{x}^{-\frac{1}{2}}$ (for $x\in\mult{K}$). Given a complex-valued function $f$ on $K$, define its $^*$-transform by \[f^*=\widehat{f\circ\nab^{-1}}\circ\nab,\] assuming that $f\circ\nab^{-1}$ is integrable.
\end{definition}

\begin{remark}
Note that the archimedean and non-archimedean $\nab$ maps have the same form: $\nab x=\sigma(x) x$ where $\sigma$ is a splitting of the absolute value.
\end{remark}

This archimedean $^*$-transform has an integral representation similar to the the Fourier transform:

\begin{lemma}\label{lemma_integral_rep_of_transform}
Let $g$ be a complex-valued function on $K$ such that $x\mapsto g(x)\abs{x}$ is integrable. Then $g^*$ is well-defined and
\[g^*(y)=2\int g(x)\psi_K(\nab(yx))\abs{x}\,dx.\]
\end{lemma}
\begin{proof}
By definition of the $^*$-transform, \[g^*(y)=\int g(u\abs{u}^{-\frac{1}{2}})\psi_K(uy\abs{y})\,du.\] To obtain the desired expression, change variables $x=u\abs{u}^{1/2}=\nab^{-1}(u)$ in the integral.
\end{proof}

\begin{remark}
The previous lemma is enough to prove that if $f$ is a Schwartz function on $K$, then both $f^*$ and $f^{**}$ are well-defined. Unfortunately, it is false that the $*$-transform of a Schwartz function is again a Schwartz function, as the following example shows.
\end{remark}

\begin{example}\label{example_*_of_schwartz}
We $^*$-transform the Schwartz function $g(x)=e^{-\pi x^2}$ on $\Real$ with additive character $e^{2\pi i x}$. Firstly, $g\circ\nab^{-1}(x)=e^{-\pi\mbox{\small sign}(x)x}$, where $\mbox{sign}(x)$ is the sign ($\pm$) of $x$, and so \[\widehat{g\circ\nab^{-1}}(y)=\int_0^{\infty}e^{-\pi y}e^{2\pi i xy}\,dx+ \int_0^{\infty}e^{-\pi y}e^{-2\pi i xy}\,dx.\]

A standard calculation from the calculus of residues is $\int_0^{\infty} e^{-\al x}e^{ibx}dx=1/(\al-ib)$ for real $\al, b$ with $\al>0$. Therefore $\widehat{g\circ\nab^{-1}}(y)=2\pi/(\pi^2+4\pi^2y^2)$ and so \[g^*(y)=\frac{2\pi}{\pi^2+4\pi^2y^4}\] which does not decay rapidly enough to be a Schwartz function. Since $g\circ\nab^{-1}$ is not differentiable at $0$, this is in agreement with the duality provided by the Fourier transform between smoothness and rapid decrease.
\end{example}

We now prove an $s$ goes to $2-s$ functional equation:
\begin{proposition} \label{proposition_archimedean_functional_equation}
Suppose that $\w$ is a quasi-character of $\mult{K}$. If $f,g$ are Schwartz functions on $K$ for which $f^*,g^*$ are also Schwartz, then \[\zeta(f,\w,s)\zeta(g^*,\w^{-1},2-s)=\zeta(f^*,\w^{-1},2-s)\zeta(g,\w,s)\] for all complex $s$. Here we write zeta functions where we strictly mean their meromorphic continuation.
\end{proposition}
\begin{proof}
One imitates Tate's method, using the representation of the $^*$-transform given by lemma \ref{lemma_integral_rep_of_transform} to show that \[\zeta(f,\w,s)\zeta(g^*,\w^{-1},2-s)=2\int\int\int_{K^3}f(x)g(z)\psi(\nab(xyz))\abs{xyz}\w(y)^{-1}\abs{y}^{-s}\,dxdydz\] for $s$ with $\realpart{s}=1-\realpart{\w}$. This expression is symmetric in $f,g$ from which follows \[\zeta(f,\w,s)\zeta(g^*,\w^{-1},2-s)=\zeta(f^*,\w^{-1},2-s)\zeta(g,\w,s).\] Apply the identity theorem to deduce that this holds for all complex $s$.
\end{proof}

\begin{example}
Suppose that $K=\mathbb{R}$; let $g(x)=e^{-\pi x^2}$, $f=g\circ\nab$. Assume that $\psi_{\mathbb{R}}(x)=e^{2\pi ix}$, and that the chosen measure is Lebesgue measure; then $\hat{g}=g$ which implies here that $f^*=f$. For $s$ complex of positive real part,
\[\zeta(f,1,s)=\frac{1}{2}\pi^{-s/4}\Gamma(s/4)=\frac{1}{2}\zeta(g,1,{s/2}).\]

The previous proposition implies that if $h,h^*$ are Schwartz on $\Real$, then 
\begin{align*}
	\zeta(h^*,1,2-s)
	&=\frac{\pi^{(s-2)/4}\Gamma((2-s)/4)}{\pi^{-s/4}\Gamma(s/4)}\zeta(h,1,s)\\
	&=2^{s/2-1}\pi^{s/2}\left(\mbox{cos}\left(\frac{\pi s}{4}\right)\Gamma\left(\frac s2\right)\right)^{-1}\zeta(h,1,s),
	\end{align*}
 by the same Gamma function identities used in \cite{Tate's-thesis}.
\end{example}

\begin{remark}
If $f$ is a Schwartz function and $\w$ a quasi-character, then we know that $\zeta(f,\w,s)/L(\w,s)$ analytically continues to an entire fnction; also, $f$ may be chosen such that $\zeta(f,\w,s)=L(\w,s)$. \emph{However}, as example \ref{example_*_of_schwartz} demonstrates, the standard choice of $f$ may be such that $f^*$ is not Schwartz.

The author suspects that if $f$ is a Schwartz function on $\Real$ for which $f^*$ is also Schwartz, then $\zeta(f,1,s)/(\pi^{-s/4}\Gamma(s/4))$ will analytically continue to an entire function; moreover, we have seen in the previous example that this denominator satisfies the "minimality" condition (ie. occurs as a zeta function). This would justify calling $\pi^{-s/4}\Gamma(s/4)$ the local L-function for $^*$.
\end{remark}
%
%
%
%

\section{Two dimensional zeta integrals} \label{section_two_dim_zeta_integrals}

In this, the final section, we apply the integration theory to the study of two-dimensional local zeta integrals.

\subsection{Non-archimedean case}
Henceforth $F$ is a two-dimensional non-archimedean local field. Thus $\Gamma=\mathbb{Z}$ and $F$ is complete with respect to the discrete valuation $\nu$, with residue field $\res{F}$ a non-archimedean (one-dimensional) local field. Let $t_1,t_2$ be local parameters for $F$ which satisfy $t_2=t(1)$ and $\res{t}_1=\pi$, where $\pi$ is the prime of $\res{F}$ used to define the $^*$-transform.

Let $K_2^t(F)$ denote the second topological $K$-group of $F$ (see \cite{Fesenko-milnor-k-groups}); recall that $K_2^t(F)$ is the appropriate object for class field theory of $F$ (see \cite{Fesenko-multidimensional-local-class-field-theory} for details). We recall those properties of $K_2^t(F)$ which we shall use:
\begin{enumerate}
\item A border map of $K$-theory defines a continuous map $\partial:K_2^t(F)\to\mult{\res{F}}$ which satisfies \[\partial\symb{u}{t_2}=
    \res{u},\quad\partial\symb{u}{v}=1\qquad(\mbox{for }u,v\in\mult{\roi{F}}).\]
$\partial$ does not depend on choice of $t_1,t_2$. Introduce an absolute value
\begin{align*}
	\abs{\cdot}:K_2^t(F)&\to\Real_{>0}\\
	\xi&\mapsto\abs{\partial(\xi)}_{\res{F}}\;.
	\end{align*}

\item Let $U$ be the subgroup of $K_2^t(F)$ whose elements have the form $\symb{u}{t_1}+\symb{v}{t_2}$, for $u,v\in\mult{\ROI{F}}$. $K_2^t(F)$ decomposes as a direct sum $\mathbb{Z}\symb{t_1}{t_2}\oplus U$. Note that $\abs{n\symb{t_1}{t_2}+u}=q^{-ns}$ for $n\in\mathbb{Z}$, $u\in U$.

\item For any quasi-character $\chi:K_2^t(F)\to\mult{\Comp}$, there exist complex $s$ and a character $\chi_0:U\to S^1$ such that \[\chi(n\symb{t_1}{t_2}+u)=\chi_0(u)q^{-ns}\qquad(\mbox{for }n\in\mathbb{Z},\;u\in U).\] The real part of $s$ is uniquely determined by $\chi$ and is said to be, as in the one-dimensional case, the exponent of $\chi$ (denoted $\realpart{\chi}$).
\end{enumerate}

\begin{definition}
Introduce $T=\mult{\roi{F}}\times\mult{\roi{F}}$, $T^+=\roi{F}\times\roi{F}$, and a surjective homomorphism
\begin{align*}
	\frakt:&T\to K_2^t(F)\\
	&(\al,\beta)\mapsto\{\al,t_2\}+\{t_1,\beta\}\end{align*}
for $\al,\beta\in\mult{\roi{F}}$.

Note that $u,v\in\mult{\ROI{F}}$ and $i,j\in\mathbb{Z}$ implies $\frakt(t_1^{i}u,t_1^{j}v)=\symb{t_1^{i}}{t_2^{j}}+\symb{t_1}{v}+\symb{u}{t_2}$.
\end{definition}

\begin{remark}
Compare with [AoAS]. $\frakt$ depends on the choice of local parameters $t_1,t_2$. $T^+$ is the closure of $T$ in the two-dimensional topology of $F$; it's relation to $T$ is the same as $\res{F}$ to $\mult{\res{F}}$ in the one-dimensional local theory, the idele group $\mult{\mathbb{A}}$ to the adele group $\mathbb{A}$ in the one-dimensional global theory, or the matrix algebra $\mbox{M}_n$ to the group $\mbox{GL}_n$ in the Jacquet-Godement generalisation \cite{Jacquet-Godement} of Tate's thesis.

Note that $(x,y)\in T$ implies $\abs{\frakt(x,y)}=\abs{x}\,\abs{y}$.
\end{remark}

Given a $\Comp(X)\;(=\CG)\;$-valued function $f$ on $T^+$, a quasi-character $\chi$ of $K_2^t(F)$, and complex $s$, Fesenko suggests in [AoAS] the following definition for the associated (two-dimensional) local zeta integral: \[\zeta(f,\chi,s)=\zeta_F^2(f,\chi,s)=\int^{\mult{F}\times\mult{F}}f(x,y)\,\chi\circ\mathfrak{t}(x,y)\abs{\mathfrak{t}(x,y)}^s\,\Char{T}(x,y)\,\dmult{x}\dmult{y},\] assuming that the integrand is integrable on $\mult{F}\times\mult{F}$ (integration on this space is discussed in appendix \ref{appendix_product_integration}).

We now prove analytic continuation, and moreover a functional equation, for a class of functions $f$ and characters $\chi$; we write $f^0$ for the lift of $f\in\calS(\res{F}\times\res{F})$ at $(0,0),(0,0)$ - see appendix \ref{appendix_product_integration}.

\begin{proposition}\label{prop_non_arch_higher_functional_eqn}
Let $\chi$ be a quasi-character of $K_2^t(F)$ and suppose that $\chi\circ\mathfrak{t}$ factors through the residue map $T\mapsto \mult{\res{F}}\times\mult{\res{F}}$. Define $L_F(\chi,s)=L(\w_1,s)L(\w_2,s)$, a product of two L-functions for $\res{F}$, and $\ep_F(\chi,s)=\ep_*(\w_1,s)\ep_*(\w_2,s)$, a product of two epsilon factors for $\res{F}$. Then 
\begin{enumerate}
\item[(AC2)] For all $f\in\calS(\res{F}\times\res{F})$, the zeta function $\zeta(f^0,\chi,s)$ is well-defined for $\realpart{s}$ sufficiently large. Moreover, \[\zeta(f^0,\chi,s)/L_F(\chi,s)\] has analytic continuation to an entire function, $Z(f^0,\chi,s)$, where $\w_i$ are the quasi-characters of $\mult{F}$ defined by $\chi\circ\mathfrak{t}(x,y)=\w_1(\res{x})\w_2(\res{y})$.
\item[(L2)] There is $f\in\calS(\res{F}\times\res{F})$ such that $Z(f^0,\chi,s)=1$ for all $s$.
\item[(FE2)] For all $f\in\calS(\res{F}\times\res{F})$, a functional equation holds: \[Z(f^{*0},\chi^{-1},2-s)=\ep_F(\chi,s)Z(f^0,\chi,s).\] for all $s$. Moreover, $\ep_F(\chi,s)$ is of exponential type; that is $\ep_F(\chi,s)=aq^{bs}$ for some complex $a$ and integer $b$.
\end{enumerate}
\end{proposition}
\begin{proof}
By definition of the integral on $\mult{F}\times\mult{F}$ and a similar argument to example \ref{example_integral_over_mult_F} (i), we have \[\zeta(f^0,\chi,s)=\int_{\mult{\res{F}}}\int_{\mult{\res{F}}} f(u,v)\w_1(u)\w_2(v)\abs{u}^s\abs{v}^s\,\dmult{u}\dmult{v},\] which we denoted $\zeta_{\res{F}\times\res{F}}(f,\w_1\otimes\w_2,s)$ in remark \ref{remark_non_arch_product_*_functional_eqn}. That is, since we are only considering functions $f$ which lift from $\res{F}\times\res{F}$, the zeta integral over $\roi{F}\times\roi{F}$ reduces to a zeta integral over $\res{F}\times\res{F}$. All required results follow from that remark.
\end{proof}

\begin{remark}
The previous example highlights the importance of lifting the $^*$-transform up to $F$ in a similar way to how we lifted the Fourier transform. Then it may be possible to generalise this proposition to many more functions on $\roi{F}\times\roi{F}$ than simply those which lift from $\res{F}\times\res{F}$.
\end{remark}

\begin{remark}
Having calculated epsilon factors for the $^*$-transformation in section \ref{section_s_goes_to_2-s}, we have formulae for the \emph{two-dimensional epsilon factors} \[\ep_F(\chi,s)=\ep_*(\w_1,s)\ep_*(\w_2,s).\] For example, if $\w_1$ is ramified with conductor $r>0$ but $\w_2$ is unramified, then \[\ep_F(\chi,s)=q^{(\lceil{(r-d)/2}\rceil-\lceil{d/2}\rceil)(2-s)}\chi(t_1,1)^{\lceil{(r-d)/2}\rceil}q^{-r/2}\delta_{d-r}\rho_0(\w_1^{-1})\] where $d$ is the conductor of the additive character on $\res{F}$ used to define the $^*$-transform.
\end{remark}

There is another relation between zeta integrals on $F$ and $\res{F}$ which we now discuss; first we need a lemma:

\begin{lemma} \label{lemma_for_lifting_zeta_integral}
Let $g$ be a complex-valued function on $\res{F}$ and $s$ complex such that $g\abs{\cdot}^{2s}$ is integrable on $\mult{\res{F}}$. Let $w:\mult{\res{F}}\to\mathbb{Z}$ be the valuation; introduce
\begin{align*}
	g':\mult{\res{F}}\times\mult{\res{F}}&\to\Comp\\
	(x,y)&\mapsto g(\pi^{\min(w(x),w(y))-w(x)}x)\,\abs{xy}^s
	\end{align*}
Then $g'$ is integrable over $\mult{\res{F}}\times\mult{\res{F}}$, with integral \[\int\int g'(x,y)\,\dmult{x}\dmult{y}=\mu(\mult{\roi{\res{F}}})\frac{1+q^{-s}}{1-q^{-s}}\int g(x)\abs{x}^{2s}\dmult{x},\] where $\mu$ is the multiplicative Haar measure on $\mult{\res{F}}$.
\end{lemma}
\begin{proof}
The integral of $g'$ over $\mult{\res{F}}\times\mult{\res{F}}$ is \[\sum_{n\in\mathbb{Z}}\sum_{m\in\mathbb{Z}} \int_{w^{-1}(n)}\int_{w^{-1}(m)} g(\pi^{\min(n,m)-m}x) q^{-s(n+m)-2}\,dxdy.\] Split the inner summation over $m<n$ and $m\ge n$ and interchange the order of the double summation $\sum_n\sum_{m<n}$; elementary manipulations complete the proof.
\end{proof}

\begin{definition}
Introduce a 'generalised residue map':
\begin{align*}
	\rho_2:&\;T^+\To\res{F}\\
	&(t_1^{i_1}t_2^{i_2}u,t_1^{j_1}t_2^{j_2}v)\mapsto\res{t_1^{\min(i_1,j_1)}t_2^{\min(i_2,j_2)}u}
	\end{align*}
where $u,v\in\mult{\ROI{F}}$ and $i_1,i_2,j_1,j_2\in\mathbb{Z}$.
\end{definition}

\begin{remark}
The map $\rho_2$, when restricted to $T$, factors through $K_2^t(F)$: \[\rho_2(t_1^iu,t_1^jv)=\partial\left( \min(i,j)\{t_1,t_2\}+\{t_1,v\}+\{u,t_2\}\right)\] where $i,j\in\mathbb{Z}$, $u,v\in\mult{\ROI{F}}$.
\end{remark}

$\rho_2$ lifts zeta integrals from $\res{F}$ to $F$:

\begin{proposition}
Let $\w$ be a quasi-character of $\mult{\res{F}}$, $s$ complex, and $g$ a complex-valued function on $\res{F}$ such that $g\,\w^2\,\abs{\cdot}^{2s}$ is integrable on $\mult{\res{F}}$; let $\chi=\w\circ\partial$. Then the zeta integral $\zeta(g\circ\rho_2,\chi,s)$ is well-defined and
\[\zeta(g\circ\rho_2,\chi,s)=\mu(\mult{\roi{\res{F}}})\frac{1+q^{-s-c}}{1-q^{-s-c}}\,\zeta_{\res{F}}(g,\w,2s+c),\] where $c\in\Comp$ is defined by $\w=\w_0\,\abs{\cdot}^{c}$ with $\w_0$ a character of $\mult{F}$ trivial on $\pi$.
\end{proposition}
\begin{proof}
For $(x,y)\in T$,
\begin{align*}
	g\circ\rho_2(x,y)\,&\chi\circ\mathfrak{t}(x,y)\,\abs{\mathfrak{t}(x,y)^s}\,\abs{x}^{-1}\,\abs{y}^{-1}\\
	&=g(\pi^{\min(w(\res{x}),w(\res{y}))-w(\res{x})}x)\,\w(\res{x}\pi^{w(\res{y})})\,\abs{\res{xy}}^{s-1}\\
	&=g(\pi^{\min(w(\res{x}),w(\res{y}))-w(\res{x})}x)\,\w_0(x)\,\abs{\res{xy}}^{s+c-1}\\
	&=g(\pi^{\min(w(\res{x}),w(\res{y}))-w(\res{x})}x)\,\w_0(\pi^{\min(w(\res{x}),w(\res{y}))-w(\res{x})}x)\,\abs{\res{xy}}^{s+c-1},\end{align*}
so that $(x,y)\mapsto g\circ\rho_2(x,y)\,\chi\circ\mathfrak{t}(x,y)\,\abs{\mathfrak{t}(x,y)^s}\,\abs{x}^{-1}\,\abs{y}^{-1}$ is the lift of \[(u,v)\mapsto g(\pi^{\min(w(u),w(v))-w(v)}u)\,\w_0(\pi^{\min(w(u),w(v))-w(u)}u)\,\abs{uv}^{s+c-1}\] at $(0,0),(0,0)$.

The result now follows from the previous lemma and appendix \ref{appendix_product_integration}.
\end{proof}

This is enough to deduce analytic continuation of some zeta functions:

\begin{corollary}
Let $\w$ be a quasi-character of $\mult{\res{F}}$, $L(\w,s)$ the associated $L$-function, and $g$ a Schwartz-Bruhat function on $\res{F}$; let $\chi=\w\circ\partial$. Then
\begin{enumerate}
\item For $s$ complex of sufficiently large real part, the zeta integral $\zeta(g\circ\rho,\chi,s)$ is well-defined.
\item The holomorphic function $\zeta(g\circ\rho,\chi,s)/(L(\w,s)(1-\w(\pi)q^{-s})^{-1})$, initially defined for $\realpart{s}$ sufficiently large, has analytic continuation to an entire function.
\end{enumerate}
\end{corollary}
\begin{proof}
This follows from the corresponding results for local zeta functions on $\res{F}$, the previous proposition, and the identity $\w(\pi)=q^{-c}$ where $c$ is as in the previous proposition.
\end{proof}

Throughout it has been useful for $\chi\circ\mathfrak{t}$ to factor through the residue map $T\mapsto\mult{\res{F}}\times\mult{\res{F}}$. In the next two examples we considers some situations in which this happens. Let $L$, a two-dimensional local field, be a finite abelian extension of $F$ and let $\chi$ be a character of $K_2^t(F)$ which vanishes on $N_{L/F}K_2^t(L)$. 

\begin{example}\label{example_chars_of_extensions_1}
Suppose $\res{L}/\res{F}$ is separable with $|\res{L}:\res{F}|=|L:F|$.

Then $\partial$ induces a surjection $K_2^t(F)/N_{L/F}K_2^t(L)\to\mult{\res{F}}/N_{L/F}\mult{\res{L}}$. Further, the separability assumption implies $\res{L}/\res{F}$ is an abelian extension of local fields so that $|\mult{\res{F}}/N_{L/F}\mult{\res{L}}|=|\res{L}:\res{F}|=|L:F|=|K_2^t(F)/N_{L/F}K_2^t(L)|$; so the induced surjection is an isomorphism. Therefore $\chi$ factors through $\partial$.
\end{example}

\begin{example}\label{example_chars_of_extensions_2}
Suppose $\res{L}=\res{F}$, $p\not|\,|L:F|$, and $t_2\in N_{L/F}\mult{L}$ ("a totally tamely ramified extension in the second parameter").

Then $(x,y)\in T$ implies $\mathfrak{t}(x,y)=\symb{t_1}{\Theta(y)}\mod{N_{L/F}K_2^t(L)}$ (see \cite{Fesenko-multidimensional-local-class-field-theory}), where $\Theta$ is the projection \[\Theta:\mult{F}=\langle t_1\rangle\times\langle t_2\rangle\times\mult{\mathbb{F}}_q\times V_F\mapsto\mult{\mathbb{F}}_q.\] Here $V_F$ is the two-dimensional groups of principal units of $F$.

Therefore there exists a tamely ramified quasi-character $\w$ of $\mult{\res{F}}$ such that $\chi\circ\mathfrak{t}(x,y)=\w(\res{y})$ for $(x,y)\in T$.
\end{example}

These examples show that our functional equation applies to all 'sufficiently unramified' characters. The proof of the functional equation in [AoAS] is valid whenever all relevant functions are integrable, and proposition \ref{prop_non_arch_higher_functional_eqn} is certainly a special case. However, it appears that if $\chi$ is ramified then certain interesting functions are not integrable. See also section $\ref{section_further_work}$.
%
%
%
%
\subsection{Archimedean case}
Now suppose that $F$ is a two-dimensional archimedean local field; that is, $\Gamma=\mathbb{Z}$ and $F$ is complete with respect to the discrete valuation $\nu$ with residue field $\res{F}$ an archimedean local field. The classification of complete discrete valuation fields (see eg \cite[chapter II.5]{FV}) implies that $F$ is isomorphic to a field of Laurent series $\Comp((t))$ or $\Real((t))$, where we write $t=t(1)$.

The correct way to use topological $K$-groups for class field theory and zeta integrals of such fields is not clear, so we content ourselves with making a few remarks about generalising the results in the non-archimedean case without appealing to $K$-groups.

Given Schwartz functions $f,g$ on $\res{F}$ for which $f^*,g^*$ are also Schwartz, and $\w$ a quasi-character of $\mult{\roi{F}}$ which factors through the residue map $\mult{\roi{F}}\to\mult{\res{F}}$, proposition \ref{proposition_archimedean_functional_equation} implies that \[\int^{\mult{F}} f^{0,0}(x)\;\w(x)\abs{x}^s\,\Char{\mult{\roi{F}}}(x)\,\dmult{x} \int^{\mult{F}} (g^*)^{0,0}(x)\;\w(x)^{-1}\abs{x}^{2-s}\,\Char{\mult{\roi{F}}}(x)\,\dmult{x}\] is invariant under interchanging $f$ and $g$. There is an analogous result for integrals over $\mult{\roi{F}}\times \mult{\roi{F}}$.

An extension of $F$ cannot be wildly ramified in any sense, and so by analogy with examples \ref{example_chars_of_extensions_1} and \ref{example_chars_of_extensions_2} we expect arithmetic characters on $\mult{\roi{F}}$ (or $\mult{\roi{F}}\times\mult{\roi{F}}$) to lift from $\mult{\res{F}}$. Hence this functional equation may be satisfactory in the archimedean case.

Indeed, in the case $F=\Comp((t))$, the finite abelian extensions of $F$ have the form $\Comp((t^{1/n}))$ for natural $n$. A character attached to such an extension is surely a purely imaginary power of $\abs{\cdot}$; this lifts to $\mult{\roi{F}}$ from $\mult{\res{F}}$.

If $F=\Real((t))$, then $F$ has maximal abelian extension $\Comp((t^{1/2}))$, with subextensions $\Real((t^{1/2}))$ and $\Comp((t))$. A character attached to the extension $\Comp((t^{1/2}))$ is $\mult{\roi{F}}\to\{\pm 1\}:x\mapsto \mbox{sg}(\res{x})$, which again lifts from $\mult{\res{F}}$.
%
%
%
%

\section{Further work}\label{section_further_work}
We discuss some topics related to the theory of this paper.

\subsection*{Ramified zeta integrals}
The proof of the functional equation in section \ref{section_two_dim_zeta_integrals} can surely be extended to a wider class of functions and characters. In particular there should be a theory for ramified characters. The results may have application in the ramification theory of two-dimensional local fields \cite{Zhukov-1} \cite{Zhukov-2}.

The author is certain that the local functional equation presented in this paper can be significantly strengthened using only the results already present in sections 1 to 6.

\subsection*{Wiener and Feynman measure; quantum field theory}
The field $\Real(t)$, and certain subspaces of $\Real((t))$, may be identified with spaces of functions. In particular, $t\Real[t]$ may be identified with a subspace of the space of continuous paths $[0,1]\to\Real$ which vanish at $0$ ie. Wiener space. It would be interesting to understand relations between Wiener measure and our measure.

Similarly, the subspace of $\Comp((t))$ consisting of Laurent series which converge on the punctured unit disc in the complex plane define continuous loops $S^1\to\Comp$. A comparison of the measure in this case, in conjunction with the real case above, may provide insight into Feynman measure on such spaces.

The values of divergent integrals in quantum field theory, after renormalisation, appear as epsilon factors in our local zeta integrals (example \ref{example_PV_archimedean}). The duality provided by a functional equation would provide arithmetic arguments for the values of such integrals. It would be very interesting to investigate whether this arithmetic value coincides with the physical value.
 
\subsection*{Other residue fields}
The philosophy followed in this paper is that any reasonable theory (eg. integration, measure theory, harmonic analysis) should 'lift' from $\res{F}$; indeed, most proofs reduce a problem on $F$ to the analogous problem on $\res{F}$, where the result is known.

If $\res{F}$ is not a local field, but instead is an infinite field with the discrete topology, then the techniques of this paper modify to lift the counting measure on $\res{F}$ to $F$. In particular, if $F$ is a complete discrete valuation field with residue field $\mathbb{F}_p^{\mbox{\tiny alg}}$ then there is a $\Comp(X)$-valued integration theory which specialises to the standard locally compact theory for any subfield $F_0\le F$ with finite residue field by taking $X=\abs{\res{F_0}}^{-1}$.

\subsection*{$F^n$ and $\mbox{GL}_n(F)$}
As discussed in appendix \ref{appendix_product_integration}, the space of Haar integrable functions (and the integral) on $\res{F}^2$ lifts to $F^2$. This space of functions is not closed under the action of $\mbox{GL}_2(F)$. There exists a different class of integrable functions which is closed under the action, and for which the identity \[\int^{F^2}f(\tau x)\,dx =\abs{\det{\tau}}^{-1}\int^{F^2}f(x)\,dx\] holds for appropriate $f$ and $\tau\in\mbox{GL}_2(F)$. See \cite{Morrow_2} for details.

Similar results hold for $F^n$, for any $n\ge 1$. Just as we deduced the existence of an invariant measure on $\mult{F}$, the results for $F^{n^2}$ imply the existence of a translation-invariant measure and integral on $\mbox{GL}_n(F)$. See \cite{Morrow_2}.

\subsection*{Non-linear change of variables and Fubini's theorem}
For applications in representation theory of algebraic groups over $F$, it is essential that the invariant measure on $F^2$ behaves well under non-linear changes of variables. For example, if $f$ is a suitable $\CG$-valued function on $F^2$ and $h$ is a polynomial with coefficients in $F$, then it is expected that \[\int^F\int^F f(x,y-h(x))\,dxdy=\int^F\int^F f(x,y-h(x))\,dydx=\int^F\int^F f(x,y)\,dydx.\]

However, recent work of the author's \cite{Morrow_3} show that this identity can fail if the local field $\res{F}$ has finite characteristic $p$. The problem may be related to the fact that the the $p^{\mbox{\tiny th}}$ powers of $\res{F}$ have zero measure in contrast to the classical case where the residue field is finite and the $p^{\mbox{\tiny th}}$ powers are therefore the entire field. Investigation of this 'measure-theoretic imperfectness' may lead to refinement of the measure.
%
%
%

\begin{appendix}
\section{Rings generated by certain classes of sets}

Certain set manipulations used in this section are inspired by [AoAS] and \cite{Halmos}.

\begin{definition}
Let $\mathcal{A}$ be a collection of subsets of some set $\Omega$.

$\mathcal{A}$  is said to be a $\emph{ring}$ if it is closed under taking differences and finite unions.

$\mathcal{A}$ is said to be a $\emph{d-class}$ if it contains the empty set and satisfies the following: $A,B$ in $\mathcal{A}$ with non-trivial intersection implies $\mathcal{A}$ contains $A\cap B$ and $A\cup B$. Elements of a d-class are called d sets.
\end{definition}

\begin{example} The following are examples of d-classes.
\begin{enumerate}
\item The collection of finite intervals of $\Real$, open on the right and closed on the left, together with the empty set.

\item The collection of translates of some chain of subgroups of a group, together with the empty set.
\end{enumerate}
\end{example}

We fix for the remainder of this appendix a d-class on some set.

\begin{lemma}\label{lemma_disjoint}
Let $A_i$ be d sets, for $i=1,\dots,n$. Then there exist disjoint d sets $B_j$, $j=1,\dots,m$ such that each $B_j$ is a union of some of the $A_i$ and such that $\Union iA=\DisjUnion jB$
\end{lemma}
\begin{proof}
A simple induction on $n$.
\end{proof}

Informally, the result states that any finite union of d sets may be refined to a disjoint union.

\begin{definition}
A set of the form $A\less\DisjUnion iA$ for some d sets $A,A_1,\dots,A_n$, with $A_i\subseteq A$ for each $i$, is said to be a $\emph{dd}$ set.
\end{definition}

\begin{remark} \label{rem_dd}
\mbox{}
\begin{enumerate}
\item Consider a set of the form $X=A\less\Union iA$ for d sets $A,A_1\dots,A_m$, where we make no assumption on disjointness or inclusions. Then $X=A\less\bigcup_iA \cap A_i$; lemma \ref{lemma_disjoint} implies that $X$ is a dd set.
\item The identity $ (A\less\DisjUnion iA) \cap (B\less\DisjUnion jB) = (A\cap B)\less(\DisjUnion iA \cup\DisjUnion jB) $ and lemma \ref{lemma_disjoint} imply that dd sets are closed under finite intersection.
\end{enumerate}
\end{remark}

\begin{definition}
A finite disjoint union of dd sets is said to be a $\emph{ddd}$ set.
\end{definition}

\begin{lemma}\label{lemma_difference_of_dd_sets}
The difference of two dd sets is a ddd set.
\end{lemma}
\begin{proof}
For arbitrary sets $A,A_0,B,(B_j)_j$ with $B_j\subseteq B$, the identity
\[(A\less A_0)\less(B\less\DisjUnion jB)=(A\less(B\cup A_0))\sqcup\bigsqcup_j((B_j\cap A)\less A_0)\]
is easily verified. Replace $A_0$ by a disjoint union of d sets and use remark \ref{rem_dd} to complete the proof.
\end{proof}

\begin{proposition}
The difference of two ddd sets is a ddd set. The union of two ddd sets is a ddd set.
\end{proposition}
\begin{proof}
The difference of two ddd sets may be written as a finite disjoint union of sets of the form $\bigcap E_i\less D_i$, a finite intersection of differences of dd sets; such a set is an intersection of ddd sets by lemma \ref{lemma_difference_of_dd_sets}. By deMorgan's laws, this may be rewritten as a disjoint union of intersections of dd sets. Hence the difference of two ddd sets is again a ddd set.

Let $D_1,\dots,D_n$ and $E_1,\dots,E_m$ be disjoint dd sets. Then $\DisjUnion iD\cup\DisjUnion jE$ is the disjoint union of the following three sets:
\begin{align*}
    W_1&=\DisjUnion iD\cap\DisjUnion jE\\
    W_2&=\DisjUnion iD\less\DisjUnion jE\\
    W_3&=\DisjUnion jE\less\DisjUnion iD\mbox{.}
\end{align*}
$W_2$ and $W_3$ are ddd sets by lemma \ref{lemma_disjoint}. Further, $W_1=\bigsqcup_{i,j}(D_i\cap E_j)$ is a ddd set by remark \ref{rem_dd}.
\end{proof}

\begin{proposition}\label{prop_ddd}
The collection of all ddd sets is a ring; indeed, it is the ring generated by the d-class.
\end{proposition}
\begin{proof}
This is the content of the previous result.
\end{proof}

\section{$\CG$-valued holomorphic functions}
We briefly explain the required theory of holomorphic functions from the complex plane to $\CG$, though $\CG$ could be replaced with an arbitrary complex vector space.

\begin{definition}
Suppose $f$ is a $\CG$-valued function defined on some open subset of the complex plane. We shall say that $f$ is holomorphic at a point of $U$ if and only if, in some neighbourhood $U_0$ of this point, \[f(z)=\sum_{i=1}^nf_i(z)p_i,\] for some $f_1,\dots,f_n$, complex-valued holomorphic functions of $U_0$, and $p_1,\dots,p_n$, elements of $\CG$.
\end{definition}

Although the definition of holomorphicity is a local one, we can find a global representation of any such function on a connected set:
\begin{proposition}
Let $(p_i)_{i\in I}$ be any basis for $\CG$ over $\Comp$, and let $(\pi_i)_{i\in I}$ be the associated coordinate projections to $\Comp$. Let $f$ be a $\CG$-valued holomorphic function on some open subset $U$ of $\Comp$. Then
\begin{enumerate}
\item $\pi_i\circ f$ is a complex-valued holomorphic (in the usual sense) function of $U$.
\item If $U$ is connected then there is a finite subset $I_0$ of $I$ and complex-valued holomorphic functions $f_i$, for $i\in I_0$, of $U$ such that \[f(x)=\sum_{i\in I_0}f_i(z)p_i\] for all $z\in U$.
\end{enumerate}
\end{proposition}
\begin{proof}
Let us suppose that \[f(z)=\sum_{j=1}^nf_j(z)q_j\tag{$\ast$}\] for all $z$ in some open $U_0\subset U$, where the $f_j$ are complex valued holomorphic functions of $U_0$ and $q_1,\dots,q_n\in\CG$. Then each $q_j$ is a linear sum (with complex coefficients) of finitely many $p_i$; therefore there is finite $I_0\subset I$ such that $f(z)=\sum_{i\in I_0} f_i(z)p_i$ for all $z\in U_0$, where each $f_i$ is a sum of finitely many $f_j$. So for any $i\in I$, \[\pi_i\circ f|_{U_0}=\begin{cases}f_i & i\in I_0\\ 0 & i\notin I_0 \end{cases}\] therefore $\pi_i\circ f$ is holomorphic on $U_0$.

But $f$ is holomorphic, so each point of $U$ has an open neighbourhood where $f$ can be written as in ($\ast$); therefore $\pi_i\circ f$ is holomorphic on all of $U$. This proves (i).

(ii) follows from (i) as soon as it is known that there are only finitely many $i$ in $I$ for which $\pi_i\circ f$ is not identically zero on $U$. But the identity theorem of complex analysis implies that if $\pi_i\circ f$ is not identically zero on $U$, then it is not identically zero on any open set $U_0\subset U$. So choose $U_0$ as at the starts of the proof and write $f|_{U_0}$ as in ($\ast$); if $\pi_i\circ f$ is not identically zero on $U_0$, then $i\in I_0$. So for all $z\in U$, \[f(z)=\sum_{i\in I_0}\pi_i\circ f(z)\,p_i.\]
\end{proof}

Although it is very easy to prove, the identity theorem here is fundamental, for else we would not be assured of the uniqueness of analytic continuations:
\begin{proposition}
Suppose that $f$ is a $\CG$-valued holomorphic function on some connected open subset $U$ of $\Comp$. Suppose that the zeros of $f$ have a limit point in $U$; then $f$ is identically zero on $U$. 
\end{proposition}
\begin{proof}
Let $(p_i)_{i\in I}$ and $(\pi_i)_{i\in I}$ be as in the previous proposition. By the usual identity theorem of complex analysis, each $\pi_i\circ f$ vanishes everywhere; therefore the same is true of $f$.
\end{proof}

Enough has now been proved to discuss analytic continuation of $\CG$-valued functions as required in section \ref{section_local_zeta_integrals}.
%
%
%

\section{Integration on $F\times F$}\label{appendix_product_integration}

In this short section we summarise without proofs the required elements of integration theory for the product space $F\times F$. Proofs of similar results may be found in \cite{Morrow_2}.

Let $\calL^2$ denote the space of Haar integrable functions on $\res{F}\times\res{F}$.

\begin{definition}
Given $g\in\calL^2$, $a_1,a_2\in F$, and $\g_1,\g_2\in\Gamma$, the lift $g$ at $(a_1,a_2),(\g_1,\g_2)$ is the function $f=g^{(a_1,a_2),(\g_1,\g_2)}$ which vanishes off $\Coset{a_1}{\g_1}\times\Coset{a_2}{\g_2}$, and satisfies \[g^{(a_1,a_2),(\g_1,\g_2)}(x_1,x_2)=g(\res{(x_1-a_1)t(-\g_1)},\res{(x_2-a_2)t(-\g_2)})\] if $x_i\in\Coset{a_i}{\g_i},\;i=1,2$. Note that if $g=g_1\otimes g_2$, where $g_i\in\calL$ for $i=1,2$, then $f=g_1^{a_1,\g_1}\otimes g_2^{a_2,\g_2}$.
\end{definition}

Let $\calL(F\times F)$ denote the $\CG$ space of $\CG$-valued functions on $F$ spanned by $g^{a,\g}$ for $g\in\calL^2$, $a\in F\times F$, and $\g\in\Gamma\times\Gamma$. The integral on $\res{F}\times\res{F}$ lifts:

\begin{proposition}
There is a (necessarily unique) $\CG$-linear functional $\int^{F^2}$ on $\calL(F^2)$ which satisfies $\int^{F^2}(g^{(a_1,a_2),(\g_1,\g_2)})=\int g(u,v)\,dudv\,X^{\g_1+\g_2}$ for $g\in\calL^2$, $a_1,a_2\in F$, $\g_1,\g_2\in\Gamma$. $\calL(F\times F)$ is closed under translation and $\int^{F^2}$ is translation-invariant.
\end{proposition}

Let $\calL(\mult{F}\times\mult{F})$ be the space of $\CG$-valued functions $\phi$ on $\mult{F}\times\mult{F}$ for which $(x,y)\mapsto \phi(x,y)\abs{x}^{-1}\abs{y}^{-1}$ may be extended to $F\times F$ as a function in $\calL(F\times F)$. Define $\int^{\mult{F}\times\mult{F}}(\phi)=\int^{F^2}\phi(x,y)\abs{x}^{-1}\abs{y}^{-1}$, where the integrand on the right is really the extension of the function to $F\times F$.

\begin{proposition}
If $\phi$ belongs to $\calL(\mult{F}\times\mult{F})$ and $\al_1,\al_2$ are in $\mult{F}$, then $(x,y)\mapsto\phi(\al_1x,\al_2y)$ belongs to $\calL(\mult{F}\times\mult{F})$ and \[\int^{\mult{F}\times\mult{F}}\phi(\al_1 x,\al_2 y)\,\dmult{x}\dmult{y}= \int^{\mult{F}\times\mult{F}}\phi(x,y)\,\dmult{x}\dmult{y}.\]
\end{proposition}
\end{appendix}

\affiliationone{
	Matthew T. Morrow\\
	Maths and Physics Building,\\
	University of Nottingham,\\
	University Park,\\
	Nottingham\\
	NG7 2RD\\
	United Kingdom\\
   	\email{matthew.morrow@maths.nottingham.ac.uk}}
\end{document}